\documentclass[12pt,
]{amsart}
\usepackage{
amssymb, graphicx
}
\usepackage{graphicx,color}
\usepackage{hyperref}
\usepackage{todonotes}
\usepackage{tikz-cd}
\newtheorem{theorem}{Theorem}
\newtheorem{corollary}[theorem]{Corollary}
\newtheorem{lemma}[theorem]{Lemma}
\newtheorem{definition}[theorem]{Definition}
\newtheorem{proposition}[theorem]{Proposition}
\newtheorem{remark}[theorem]{Remark}

\theoremstyle{remark}

\date{\today}

\DeclareMathOperator{\grad}{grad}

\newcommand{\eps}{\varepsilon}
 
\newcommand{\R}{{\bf R}}
\newcommand{\dM}{\partial M}
\newcommand{\din}{\partial_{in}}

\newcommand{\upo}{U_{p_0}}
\newcommand{\vzo}{V_{p_0}}
\newcommand{\exit}{{\tau_{\text{exit}}}}
\newcommand{\cut}{{\tau_{\text{cut}}}}
\newcommand{\con}{{\tau_{\text{con}}}}
\newcommand{\locus}{{\text{cut}}}

\newcommand{\be}[1]{\begin{equation}\label{#1}} 
\newcommand{\ee}{\end{equation}} 

\newcommand{\N}{\mathbf{N}}

\newcommand{\p}{\partial}

\newcommand{\Pwave}{\textit{P}}
\newcommand{\Swave}{\textit{S}}

\title[Boundary distance representation with partial data]{Uniqueness of the partial travel time representation of a compact Riemannian manifold with strictly convex boundary}

  \author[E. Pavlechko]{Ella Pavlechko}
\address{Department of Mathematics,
  North Carolina State University, Raleigh, NC 27695, USA
   (\tt{epavlec@ncsu.edu})}

\author[T. Saksala]{Teemu Saksala}
\address{Department of Mathematics,
  North Carolina State University, Raleigh, NC 27695, USA
   (\tt{tssaksal@ncsu.edu})}


 
\begin{document}

\begin{abstract}
In this paper a compact Riemannian manifold with strictly convex boundary is reconstructed from its partial travel time data. This data assumes that an open measurement region on the boundary is given, and that for every point in the manifold, the respective distance function to the points on the measurement region is known. This geometric inverse problem has many connections to seismology, in particular to microseismicity. The reconstruction is based on embedding the manifold in a function space. This requires the differentiation of the distance functions. Therefore this paper also studies some global regularity properties of the distance function on a compact Riemannian manifold with strictly convex boundary. 
\end{abstract}

\maketitle

\section{Introduction}
\label{sec:intro}
This paper is devoted to an inverse problem for smooth compact Riemannian manifolds with  smooth  boundaries.  Suppose that there is a dense but unknown set of point sources going off in an unknown Riemannian manifold. For each of these point sources we measure the travel time of the respective wave on some small known open subset of the boundary of the manifold. The point source can be natural (e.g. an earthquake as a source of seismic waves) or artificial (e.g. produced by focusing of waves or by a wave sent in scattering from a point scatterer). We show that this partial travel time data determines the unknown Riemannian manifold up to a Riemannian isometry under some geometric constraints. Namely, we assume that the unknown manifold has a strictly convex boundary. 
We also provide an example demonstrating that the convexity is a necessary assumption. To prove our result, we embed a Riemannian manifold with boundary into a function space and use smooth boundary distance functions to give a coordinate and the Riemannian structures. 
 
The inverse problem introduced above can be rephrased as the following
problem in seismology. Imagine that earthquakes occur at known times
but unknown locations within Earth's interior and arrival times are
measured on some small area on the surface. Are such partial travel time measurements sufficient to determine the possibly anisotropic elastic wave speed everywhere in the interior and pinpoint the locations of the
earthquakes? While earthquake times are not  known in practice, this
is a fundamental mathematical problem that underlies more elaborate
geophysical scenarios.

An elastic body --- e.g. a planet --- can be modeled as a manifold, where distance is measured in travel time:
The distance between two
points is the shortest time it takes for a wave to go from one point
to the other.
If the material is isotropic or elliptically anisotropic, then this
elastic geometry is Riemannian.
However this sets a very stringent
assumption on the stiffness tensor describing the elastic system, one which is not physically observed in the Earth, and
Riemannian geometry is therefore insufficient to describe the
propagation of seismic waves in the Earth.
If no structural
assumptions on the stiffness tensor apart from the physically
necessary symmetry and positivity properties are made, this leads
necessarily to modeling the planet by a Finsler manifold as was explained in~\cite{de2021determination}. 

An isotropically elastic medium carries pressure (\Pwave) and shear (\Swave) wave speeds that are conformally Euclidean metrics.
Of these two the \Pwave-waves are faster~\cite{cerveny2005seismic}.
In order to be true to the isotropic elasticity we should measure both \Pwave- and \Swave-wave arrivals.
In this paper we simplify this aspect of the problem by disregarding polarizations and considering only one type of isotropic waves.

\subsection*{Acknowledgements} EP was funded by the AWM and NSF-DMS grant \# 1953892 for travel to the 2022 Joint Mathematics Meetings (JMM), where this paper is to be presented.

\subsection{Main theorem and the geometric assumptions}
\label{sec:main_thm}
We consider a compact $n$-dimen-sional smooth manifold $M$ with smooth boundary $\dM$, equipped with a smooth Riemannian metric $g$.  For points  $p,q \in M$ the Riemannian distance between them is denoted by $d(p,q)$. 
Then for $p \in M$ we define the \textit{boundary distance function} $\hat{r}_p:\dM \to \R$ given by $\hat{r}_p(z) = d(p,z)$. Let $\Gamma$ be a non-empty open subset of the boundary $\dM$. We denote the restriction of the boundary distance function on this set as $r_p:=\hat{r}_p\big|_\Gamma$. The collection
\begin{equation}
\Gamma \qquad \text{and} \qquad \{ r_p: \Gamma \to \R ~:~ r_p(z) = d(p,z) \},
\label{eqn:data}
\end{equation}
are called the \textit{partial travel time data of}  $\Gamma\subset \dM$. With these data we seek to recover the Riemannian manifold $(M,g)$ up to a Riemannian isometry. The following definition explains when two Riemannian manifolds have the same partial travel time data \eqref{eqn:data}.

\begin{definition}
\label{def:data}
Let $(M_1, g_1)$ and $(M_2,g_2)$ be compact, connected, and oriented Riemannian manifolds of dimension $n\in \N, \: n\geq 2$ with smooth boundaries $\dM_1$ and $\dM_2$ and open non-empty regions $\Gamma_i\subset \dM_i$ respectively. We say that the partial travel time data of $(M_1, g_1)$ and $(M_2,g_2)$ coincide if there exists a diffeomorphism $\phi:\Gamma_1 \to \Gamma_2$ such that
\begin{equation}
\label{eq:equiv_of_data}
\{r_p \circ \phi^{-1}~:~ p \in M_1\} = \{ r_q ~:~ q \in M_2\}.
\end{equation}
\end{definition}
\noindent
We want to emphasize that the equality \eqref{eq:equiv_of_data} is for the non-indexed sets of travel time functions. Thus, for any $p \in M_1$ there exists a point $q \in M_2$ such that $r_{p}(\phi^{-1}(z)) = r_{q}(z)$ for every $z \in \Gamma_2$. We do not know \textit{a priori} where the point $p\in M_1$ is or if there are several points $q\in M_2$ that satisfy this equation.

\medskip

We use the notations $TM$ and $SM$ for the tangent and unit sphere bundles of $M$. Their respective fibers, for each point $p \in M$, are denoted by $T_pM$ and $S_pM$. In order to show that the data \eqref{eqn:data} determine $(M,g)$, up to an isometry or in other words that the Riemannian manifolds $(M_1,g_1)$ and $(M_2,g_2)$ of Definition \ref{def:data} are Riemanian isometric, we need to place an additional geometric restriction. We assume that $(M,g)$ has a \textit{strictly convex boundary} $\dM$ which means that the shape operator $S\colon T\p M \to T\p M$ as a linear operator on each tangent space $T_x\dM$ of the boundary $\dM$ for a point $x \in \dM$ is negative definite. This means that the second fundamental form
\[
\Pi_{x}(X,Y)=\langle S(x) X, Y\rangle_g,
\]
is strictly negative whenever $X,Y\in T_x\p M$ agree, but do not vanish.  
It has been shown in \cite{bishop1974infinitesimal, bartolo2011convex} that the strict convexity of the boundary implies the \textit{geodesic convexity} of $(M,g)$. That is any pair of points $p,q \in M$ can be connected by a distance minimizing geodesic (not necessarily unique) which is contained in the interior $M^{int}$ of $M$ modulo the terminal points. In particular any geodesic of $M$ that hits the boundary exits immediately. 

The main theorem of this paper is the following:

\begin{theorem} 
Let $(M_1, g_1)$ and $(M_2,g_2)$ be compact, connected, and oriented Riemannian manifolds of dimension $n\in \N, \: n\geq 2$ with smooth and strictly convex boundaries $\dM_1$ and $\dM_2$ and open non-empty measurement regions $\Gamma_i\subset \dM_i$ respectively.
If the travel time data of $(M_1,g_1)$ and $(M_2,g_2)$ coincide, in the sense of Definition \ref{def:data}, then the Riemannian manifolds $(M_1,g_1)$ and $(M_2,g_2)$ are Riemannian isometric. 
\label{thm:main}
\end{theorem}

\begin{remark}
Our assumptions in Theorem \ref{thm:main} do not prevent the existence of the conjugate points. Actually quite a lot of work in this paper is needed to handle their existence. We also allow the manifolds to have trapped geodesics.
\end{remark}

\subsection{Outline of the proof of Theorem \ref{thm:main}}
The main tool of proving Theorem \ref{thm:main} is to differentiate the travel time functions given in equation \eqref{eqn:data}. As these functions are defined only on a small open subset of the boundary we need to develop some regularity theorem for them. For this reason in Section \ref{sec:dist_funct} we study the regularity properties of the distance function on Riemannian manifolds satisfying the geometric constraints of Theorem \ref{thm:main}. Section \ref{sec:dist_funct} has two main results. Theorem \ref{prop:dist_smoothness} is the aforementioned regularity result and the key of the proof of Theorem \ref{thm:main}. In order to prove Theorem \ref{prop:dist_smoothness} we need to study, for each point in our manifold, the properties of its cut locus. This is the set past which the geodesics shot from the chosen point are not anymore distance minimizers. Theorem \ref{thm:cut_locus} collects the needed properties of these sets. Up to the best of our knowledge the material presented in Section \ref{sec:dist_funct} does not exist or is not easily accessible in the literature. Nevertheless, the corresponding results for manifolds without boundaries are well known. 

In Section \ref{sec:reconst} we apply Theorem \ref{prop:dist_smoothness} to reconstruct the Riemannian manifold from its partial travel time data \eqref{eqn:data}. This is done in five parts. Firstly we recover the geometry of the measurement region. As the second step we recover the topological structure by embedding the unknown manifold into a function space. Then we determine the boundary. The fourth step is to find local coordinates. Since our manifold has a boundary, we need different types of local coordinates for the interior and boundary points. Lastly we reconstruct the Riemannian metric. All the steps in Section \ref{sec:reconst} are fully data driven. Finally, in Section \ref{sec:proof} we show that if two Riemannian manifolds, as in Theorem \ref{thm:main}, have coinciding partial travel time data, in the sense of the Definition \ref{def:data}, then they are isometric. 
 

\subsection{The convexity of the domain in Theorem \ref{thm:main} is necessary}
 Let us construct an explicit example of a surface~$M$ and a subset $\Gamma\subset\partial M$ so that our results fail with data recorded only on~$\Gamma$ (this example was originally presented in \cite{de2021stable}). 
We recall that every pair of points on a smooth compact Riemannian manifold with boundary is always connected by a $C^1$-smooth distance minimizing curve~\cite{alexander1981geodesics}.
We choose our a manifold to be the horseshoe-shaped domain of Figure~\ref{Fi:f_p}. We split the domain~$M$ into two pieces~$M_1$ and~$M_2$ with respect to the line (red dotted line) that is normal to~$\p M$ at $x_0\in \p M$ (blue dot).
Then we choose a domain  $\Gamma \subset \p M_1$ (red arch) so that any minimizing curve joining a point on~$\Gamma$ and a point in~$M_2$ touches the boundary near~$x_0$.
The curve $P\subset M$ is any involute of the boundary, meaning that the distance from all points on~$P$ to~$x_0$ is the same. Because $d(z,p)=d(z,q)$ for any $z\in\Gamma$ and $p,q\in P$, from the point of view of our data \eqref{eqn:data}, the set~$P$ appears to collapse to a point.

\begin{figure}[ht]
\begin{picture}(300,200)
  \put(0,0){\includegraphics[width=10cm]{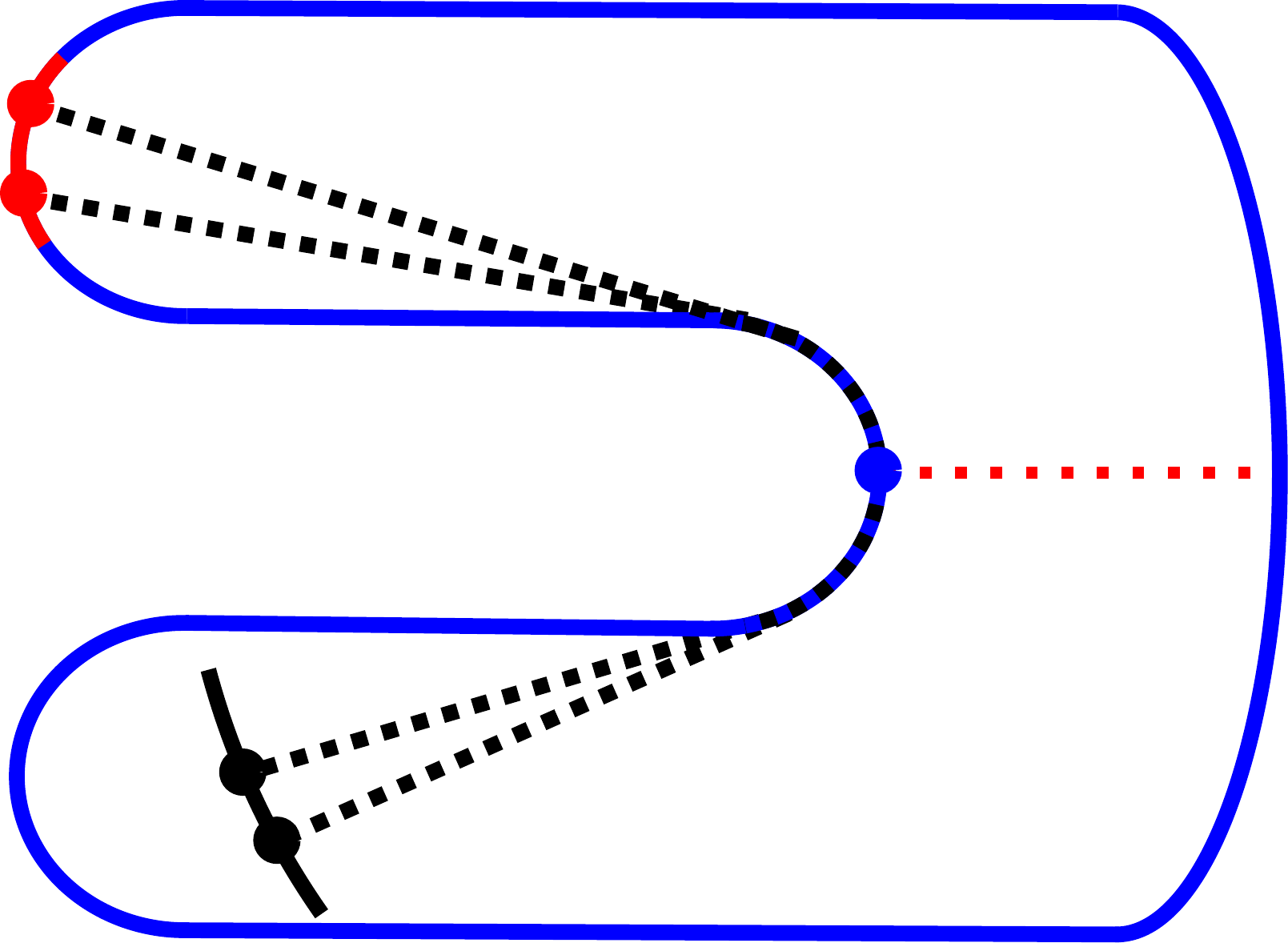}}
  \put(35,25){$P$}
  \put(175,15){$M_2$}
  \put(175,100){$x_0$}
    \put(-15,170){$\Gamma$}
    \put(175,175){$M_1$}
  \end{picture}
\caption{A domain where partial data is insufficient.}

\label{Fi:f_p}
\end{figure}
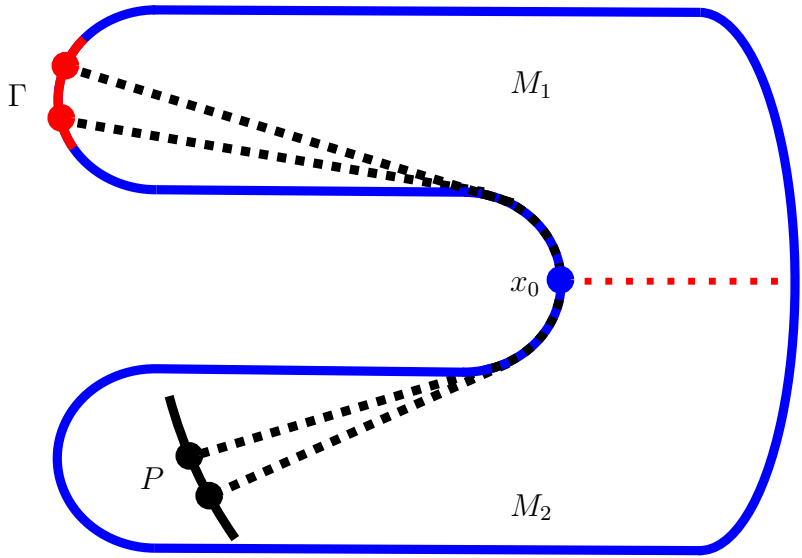

\section{Some geometric inverse problems arising from seismology}
In the following subsections we review some seismologically relevant geometric inverse problems on Riemannian and Finsler manifolds.

\subsection{Related geometric inverse problems on Riemannian manifolds}
Let $(M,g)$ be as in Theorem \ref{thm:main} and denote the Laplace-Beltrami operator of the metric $g$ by $\Delta_g$. Let $(p,t_0)\in M^{int}\times \R$. In this paper we consider an inverse problem related to the wave equation 
\begin{equation}
    \begin{cases}
    (\partial_t^2 - \Delta_g) u(x,t) = \delta_p(x) \delta_{t_0}(t),\\
    u(x,t)= 0,  \qquad t< t_0, ~x \in M,
    \end{cases}
    \label{eqn:ivp}
\end{equation}
where the solution $u(x,t)$ is a spherical wave produced by an interior point source given by the delta function $\delta_p(x) \delta_{t_0}(t)$ of the space time $M \times \R$ at $(p,t_0)$. We define the \textit{arrival time function} $\mathcal T_{p,t_0} \colon \p M \to \R$, of the wave $u$ by the formula
\[
\begin{split}
\mathcal T_{p,t_0}(z)=\sup \{& t\in \R:\hbox{ the point $(z,t)\in \p M \times \R$ has a neighborhood}
\\
&  \text{$U\subset M \times \R$ such that $ u(\cdot,\cdot)\big|_U =0$}\}.
\end{split}
\]
We recall that it was shown in \cite[Proposition 3.1]{lassas2015determination}, by applying the results of \cite{duistermaat1972fourier, greenleaf1993recovering}, that 
$
\mathcal T_{p,t_0}(z)=d(p,z)+t_0.
$
If the initial time $t_0$ is zero (or given), the knowledge of the arrival time function $\mathcal T_{p,t_0}$ on the open set $\Gamma \subset \p M$ is equivalent to the corresponding travel time function $r_p(\cdot)$ as in \eqref{eqn:data}. 

The problem of determining the isometry type of a compact Riemannian manifold from its \textit{boundary distance data}
\[
\{\hat r_p\colon \p M \to (0,\infty);\: \hat{r}_p(z)=d(p,x),\: p \in M^{int}\}
\]
was introduced for the first time in~\cite{kurylev1997multidimensional}. The reconstructions of the smooth atlas on the manifold and the metric tensor in these coordinates was originally considered in~\cite{Katchalov2001}. In contrast to the paper at hand, these earlier results do not need any extra assumption for the geometry, but have complete data in the sense that the measurement area $\Gamma$ is the whole boundary. Their counterpart of our Theorem \ref{prop:dist_smoothness}, is \cite[Lemma 2.15]{Katchalov2001}, which says that the boundary distance function $\hat r_p$ is smooth near its minimizers (the set of closest boundary points to the source point $p\in M$). This lemma is a key of their proof. However, the same technique is not available to us as it requires the access to the whole boundary. 

The problem of boundary distance data is related to many other geometric inverse problems. For instance, it is a crucial step in proving uniqueness for Gel'fand's inverse boundary spectral problem~\cite{Katchalov2001}.
Gel'fand's problem concerns the question whether the data
\[
(\p M, (\lambda_j, \p_\nu \phi_j|_{\p M})_{j=1}^\infty)
\]
determine $(M,g)$ up to isometry, when $(\lambda_j, \phi_j)$ are the Dirichlet eigenvalues and the corresponding $L^2$-orthonormal eigenfunctions of the Laplace--Beltrami operator.
Belishev and Kurylev provide an
affirmative answer to this problem in~\cite{belishev1992reconstruction}.

In \cite{katsuda2007stability} the authors studied a question of approximating a Riemannian manifold under the assumption: For a finite set of receivers $R\subset \p M$ one can measure the travel times $d(p,\cdot)|_{R}$ for finitely many $p \in P\subset  M^{int}$ under the \textit{a priori} assumption that $R \subset \p M$ is $\eps$-dense and that $\{d(p,\cdot)|_{R}: \: p \in P\}\subset \{d(p,\cdot)|_{R}: \: p \in M^{int}\}$ is also $\eps$-dense.  Thus $\{d(p,\cdot)|_{R}: \: p \in P\}$ is a finite measurement.  The authors construct an approximate finite metric space  $M_\eps$ and show that the Gromov-Hausdorff distance of $M$ and $M_\eps$ is proportional to some positive power of $\eps$. In \cite{katsuda2007stability} an independent travel time measurement is made for each interior source point in $P$, whereas in \cite{de2021stable} the authors studied the approximate reconstruction of a simple Riemannian manifold (a compact Riemannian manifold with strictly convex boundary where each pair of points is connected by the unique smoothly varying distance minimizing geodesic) by measuring the arrival times of wave fronts produced by several point sources, that go off at unknown times, and moreover, the signals from the different point sources are mixed together.  To describe the similarity of two metric spaces `with the same boundary' the authors defined a labeled Gromov-Hausdorff distance. This is an extension of the classical Gromov-Hausdorff distance which compares both the similarity of the metric spaces and the sameness of the boundaries --- with a fixed model space for the boundary. In addition to reconstructing a discrete metric space approximation of $(M,g)$, the authors in \cite{de2021stable} estimated the density of the point sources and established an explicit error bound for the reconstruction in the labeled Gromov-Hausdorff sense.

If we do not know the initial time~$t_0$ in \eqref{eqn:ivp}, but we can recover the arrival times $\mathcal{T}_{p,t_0}(z)$ for each $z\in \p M$, then taking the difference of the arrival times one obtains a boundary distance difference function
\[
D_{p}(z_1,z_2):=d(p,z_1)-d(p,z_2)
\]
for all $z_1,z_2 \in \p M$,
which is independent from the initial time $t_0$.
In~\cite{lassas2015determination}
it is shown that if $U\subset N$ is a compact subset of a closed
Riemannian manifold $(N,g)$ with a non-empty interior, then
\emph{distance difference data} 
\[
((U,g|_{U}), \{D_p\colon U\times U
\to \R \:| \:p \in N\})
\]
determine $(N,g)$ up to an isometry. This result was generalized for complete Riemannian manifolds~\cite{ivanov2018distance} and for compact Riemannian manifolds with boundary~\cite{de2018inverse, ivanov2020distance}. These results require the full boundary measurement in the sense of $\Gamma=\p M$, unlike Theorem \ref{thm:main} in the present paper. 

If the sign in the definition of the distance difference functions is
changed, we arrive at the distance sum functions,
\[
D^+_p(z_1,z_2)=d(z_1,p)+d(z_2,p)
\]
for all $p\in M$ and $z_1,z_2\in \p M$.
These functions give the lengths of the broken geodesics, that is, the
union of the shortest geodesics connecting~$z_1$ to~$p$ and the
shortest geodesics connecting~$p$ to~$z_2$.
Also, the gradients of $D^+_p(z_1,z_2)$ with respect to~$z_1$ and $z_2$ give the velocity
vectors of these geodesics.
The inverse problem of determining the
manifold $(M,g)$ from the \emph{broken geodesic data}, consisting of
the initial and the final points and directions, and the total length
of the broken geodesics, has been considered in~\cite{kurylev2010rigidity}. The authors show that broken geodesic data determine the boundary distance data of any compact smooth manifold of dimension three and higher. Finally they use
the results of \cite{Katchalov2001, kurylev1997multidimensional} to prove that the broken geodesic data determine the Riemannian manifold up to an isometry. 
A different variant of broken geodesic data was recently considered in~\cite{meyerson2020stitching}.

The Riemannian wave operator is a globally hyperbolic
linear partial differential operator of real principal
type. Therefore, the Riemannian distance function and the propagation
of a singularity initiated by a point source in space time are related
to one another.
We let~$u$ be the solution of the Riemannian wave equation as in \eqref{eqn:ivp}.
In \cite{duistermaat1996fourier, greenleaf1993recovering} it is shown
that the image, $\Lambda$, of the wave front set of~$u$, under the
musical isomorphism $T^\ast M \ni (x,\xi) \mapsto (x,g^{ij}(x)\xi_i) \in TM$, coincides
with the image of the tangent space ~$T_{p}M$ at~$p \in M^{int}$ 
under the geodesic flow of~$g$. Thus $\Lambda \cap \p(S M)$, where~$SM$ is the unit sphere bundle of $(M,g)$, coincides with the exit
directions of geodesics emitted from~$p$.
In~\cite{lassas2018reconstruction} the authors show that if $(M,g)$ is a
compact smooth non-trapping Riemannian manifold with smooth strictly
convex boundary, then generically the \emph{scattering data of point 
sources} $(\p M, R_{\p M}(M))$ determine $(M,g)$ up to an isometry.
Here, $R_{\p M}(p) \in R_{\p M}(M)$ for $p \in M$ stands for
the collection of tangential components to the boundary of exit directions
of geodesics from~$p$ to~$\p M$. 


A classical geometric inverse problem, that is closely related to the distance functions, asks:
Does the Dirichlet-to-Neumann mapping of a Riemannian wave operator determine a Riemannian manifold up to an isometry?
For the full boundary data this problem was solved originally in~\cite{belishev1992reconstruction} using the Boundary control method. Partial boundary data questions have been studied for instance in \cite{lassas2014inverse, milne2019codomain}.
Recently~\cite{kurylev2018inverse} extended these results for connection Laplacians.
Lately also inverse problems related to non-linear hyperbolic equations have been studied extensively  \cite{kurylev2017inverse, lassas2018inverse, wang2016inverse}.
For a review of inverse boundary value problems for partial differential equations see \cite{LassasICM2018, uhlmann1998inverse}.

Maybe the most extensively studied geometric inverse problem formulated with the distance functions is the Boundary rigidity problem. This problem asks:  Does the \emph{boundary distance function}, that gives a distance between any two boundary points, determine $(M,g)$ up to an isometry? In an affirmative case $(M,g)$ is said to be boundary rigid. For a general Riemannian manifold the problem is false: Suppose the manifold contains a domain with very slow wave speed, such that all the geodesics starting and ending at the boundary avoid this domain. Then in this domain one can perturb the metric in such a way that the boundary distance function does not change. It was conjectured in~\cite{michel1981rigidite} that for all simple Riemannian manifolds the answer is affirmative. In two dimensions this was verified in~\cite{pestov2005two}. For higher dimensional cases the problem is still open, but different variations of it has been considered for instance in \cite{burago2010boundary, croke1991rigidity, stefanov2016boundary, stefanov2017local}. 

\subsection{Related geometric inverse problems on Finsler manifolds}

In~\cite{de2021determination} the authors studied the recovery of a compact Finsler manifold from its boundary distance data. In contrast to earlier Riemannian results \cite{Katchalov2001, kurylev1997multidimensional} the data only determines the topological and smooth structures, but not the global geometry. However the Finsler function $F\colon TM \to [0,\infty)$ can be recovered in a closure of the set $G(M,F)\subset TM$, which consists of points $(p,v)\in TM$ such that the corresponding geodesic~$\gamma_{p,v}$ is distance minimizing to the terminal boundary point. They also showed that if the set $TM \setminus G(M,F)$ is non-empty then any small perturbation of~$F$ in this set leads to a Finsler metric whose boundary distance data agrees with the one of~$F$. If $G(M,F)=TM$, then the boundary distance data determines $(M,F)$ up to a Finsler isometry. For instance the isometry class of any simple Finsler manifold is determined by this data.
The same is not true if only the boundary distance function is known~\cite{ivanov2013local}. Thus a simple Finsler manifold is never boundary rigid. In~\cite{de2020foliated} the main result of~\cite{de2021determination} was utilized to generalize the result of~\cite{kurylev2010rigidity}, about the broken geodesic data, on reversible Finsler manifolds, satisfying a convex foliation condition. 

Although, simple Finsler manifolds are not boundary rigid there are results considering their rigidity questions for some special Finsler metrics.
For instance it was shown in~\cite{monkkonen2020boundary} that Randers metrics $F_k=G_k+\beta_k$ indexed with $k \in \{1,2\}$ with simple and boundary rigid Riemannian norm $G_k(x,v)=\sqrt{g_{ij}(x)v^iv^j}$ and closed one-form~$\beta_k$, have the same boundary distance function if and only if $G_1=\Psi^{\ast}G_2$ for some boundary fixing diffeomorphism $\Psi\colon M \to M$ and $\beta_1-\beta_2=\mathrm d \phi$ for some smooth function~$\phi$ vanishing on~$\p M$. It is worth mentioning that analogous results have been presented earlier on a Riemannian manifold in the presence of a magnetic field \cite{AsZh, dairbekov2007boundary}.

\section{Distance functions on compact manifolds with strictly convex boundary}
\label{sec:dist_funct}

The aim of this section is to prove the following regularity result for the Riemannian distance function. 

\begin{theorem}
Let $(M,g)$ be a smooth, compact, connected, and oriented Riemannian manifold of dimension $n\in \N, \: n\geq 2$ with smooth and strictly convex boundary. For any $p_0 \in M$ there exists an open and dense set $W_{p_0}\subset \p M$ such that for every $z_0 \in W_{p_0}$ there are neighborhoods $\upo\subset M$ of $p_0$ and $\vzo\subset M$ of $z_0$ such that the distance function $d(\cdot,\cdot)$ is smooth in the product set $\upo\times \vzo$.
\label{prop:dist_smoothness}
\end{theorem}

This result is the key of the proof of Theorem \ref{thm:main}.

\subsection{Critical distances, extensions and the cut locus}
In this section we consider a Riemannian manifold $(M,g)$ as in Theorem \ref{prop:dist_smoothness}, and study the properties of several critical distance functions. We define the \textit{exit time function}
\[
\exit\colon SM \to \R\cup\{\infty\}, \quad \exit(p,v)=\sup\{t>0: \gamma_{p,v}(t) \in M^{int}\},
\]
where $\gamma_{p,v}$ is the geodesic of $(M,g)$ with the initial conditions $(p,v)\in SM$.  Since the boundary of $M$ is strictly convex, $\exit(p,v)$ is the first time when the geodesic $\gamma_{p,v}$ hits the boundary, and $(-\exit(p,-v),\exit(p,v))$ is the maximal interval where the geodesic is defined. We do not assume that $\exit(p,v)<\infty$ for all $(p,v)\in SM$. That is, $(M,g)$ may have trapped geodesics.
Here we denote by $J \subset SM$ the set of all non-trapped directions, that are those $(p,v) \in SM$ for which $\exit(p,v) <\infty$. It is well known that on compact Riemannian manifolds with strictly convex boundary the set $J$ is open in $SM$, the exit time function $\exit$ is continuous in $J$, and smooth on $J\setminus T\dM$ (See for instance \cite[Chapter 4]{Shara}).

For any $p\in M$ we define a star shaped set 
\begin{equation}
    \label{eq:M_p_set}
M_p:=\left\{v \in T_pM: \: v=0, \: \text{ or } \|v\|_{g}\leq \exit\left(p,\frac{v}{\|v\|_{g}}\right)\right\}.
\end{equation} 
Thus $M_p$ is the largest subset of $T_pM$ where the exponential map of $p$
\[
\exp_p\colon M_p \to M, \quad \exp_p(v)=\gamma_{p,v}(1)
\]
is defined. Since $\p M$ is strictly convex this map is onto, but it does not need to be one-to-one, since there can be several geodesics of the same length connecting $p$ to some common point. This leads to the following definition of the \textit{cut distance} function:
\begin{equation}
\label{eq:cut_dist_func}
\cut \colon SM \to \R,\quad \cut(p,v)=
\sup\{t \in (0,\exit(p,v)] : \: d(p,\gamma_{p,v}(t))=t\}.
\end{equation}
Thus the geodesic segment $\gamma_{p,v}\colon [0,t] \to M$ is a distance minimizing curve for any $t\in[ 0,\cut(p,v)]$.

Traditionally on a closed Riemannian manifold $(N,g)$ the set 
\begin{equation}
    \label{eq:cut_loci}
\locus_{N}(p):=\{\gamma_{p,v}(\cut(p,v)) \in N:  \: v \in S_pN\}
\end{equation}
is known as the \textit{cut locus} of the point $p \in N$ and each point in this set is called a \textit{cut point} of $p$. Moreover, the cut locus of $p$ coincides with the closure of the set of those points $q \in N$ such that there is more than one distance minimizing geodesic from $p$ to $q$ (see for instance \cite[Theorem 2.1.14]{klingenberg}). It has been also shown in \cite[Section 9.1]{petersen2006riemannian} that $d(p,\cdot)$ is smooth in $N\setminus (\{p\}\cup\locus_N(p))$ but not at any $q \in (\{p\}\cup\locus_N(p))$. In order to understand the smoothness properties of the distance function on a Riemannian manifold $(M,g)$ with a strictly convex boundary, our aim is to define the set analogous to the one in \eqref{eq:cut_loci} in this context. 

If $N$ is a closed manifold and $(p,v)\in SN$ then by Klingenberg's lemma \cite[Proposition 10.32]{lee2018introduction} either there is a second distance minimizing geodesic from $p$ to $\gamma_{p,v}(\cut(p,v))$ or these points are conjugate to each other along $\gamma_{p,v}$. In particular, the geodesic $\gamma_{p,v}$ is not a distance minimizer beyond the interval $[0,\cut(p,v)]$. The following lemma extends this result in our case.

\begin{lemma}
Let Riemannian manifold $(M,g)$ be as in Theorem \ref{prop:dist_smoothness} and  $(p,v)\in SM$. If  
\[
\cut(p,v) <  \exit(p,v)
\]
then at least one of the following holds for $q:=\gamma_{p,v}(\cut(p,v))$: 
\begin{itemize}
\item There exists another distance minimizing geodesic from $p$ to $q$.
\item $q$ is the first conjugate point to $p$ along $\gamma_{p,v}$.
\end{itemize}

Moreover, for any $t_0\in (0,\cut(p,v))$ the geodesic segment $\gamma_{p,v}\colon [0,t_0]\to M$ has no conjugate points and is the unique unit-speed distance minimizing curve between its endpoints. 

\label{lem:cut_time}
\end{lemma}

\begin{proof}
Since the exit time function is continuous on the 
non-trapping part of $SM$ 
and $q$ is an interior point of $M$, the proof is identical to the proof of the analogous claim in \cite[Proposition 10.32]{lee2018introduction}. 
\end{proof}

Since the manifold $M$ has a non-empty boundary $\p M$ it holds that both the tangent bundle $TM$ and the unit sphere bundle $SM$ are manifolds with boundaries $\p TM $ and $\p SM$ respectively. 
\[
(p,v)\in \p TM, \: ((p,v)\in \p SM) \quad \hbox{ if and only if } \quad p \in \p M. 
\] 
We equip $TM$ with the Sasaki metric $g_S$. Thus we can consider $TM$, and its submanifold $SM$, as Riemannian manifolds. In the following the convergence and other metric properties in $TM$ or $SM$ will be considered with respect to this metric.

\begin{lemma}
\label{lem:cut_dis_is_cont}
Let Riemannian manifold $(M,g)$ be as in Theorem \ref{prop:dist_smoothness}. The cut distance function $\cut$ is continuous in $SM$.
\end{lemma}

\begin{proof}
Since the exit time function is continuous
on the non-trapping part of $SM$,
the proof of this claim is almost identical to the proof of the analogous claims in \cite[Theorem 10.33]{lee2018introduction} and \cite[Lemma 2.1.5]{klingenberg}.
\end{proof}

As $M$ has a boundary, the definition of the cut time function $\cut$, in the equation \eqref{eq:cut_dist_func}, has an issue. Namely if $\cut(p,v)=\exit(p,v)$ for some $(p,v)\in SM$ we do not know \textit{a priori} if the geodesic $\gamma_{p,v}$ just hits the boundary at $\gamma_{p,v}(\cut(p,v))$ or if it is possible to find an extension of $(M,g)$ such that $\gamma_{p,v}$ also extends as a distance minimizer. 

To address this question, from here onwards we assume that $(M,g)$ has been isometrically embedded in some closed Riemannian manifold $(N,g)$. This can be done for instance by constructing the double of the manifold $M$ as explained in \cite[Example 9.32]{lee2013smooth} and extending the metric $g$ smoothly across the boundary $\p M$. The issue with this extension is that it might create `short cuts' in the sense that there can be a curve in $N$, connecting some points of $M$, which is shorter than any curve entirely contained in $M$. Therefore we always have
\[
d_{M}(p,q)\geq d_{N}(p,q), \quad \hbox{ for all } p,q \in M,
\]
where $d_{M}(\cdot,\cdot)$ and $d_N(\cdot,\cdot)$ are the distance functions of $M$ and $N$ respectively. The following proposition shows that while we stay close enough to $M$ we do not need to worry about these short cuts.

\begin{proposition}
\label{prop:extension}
Let $(N,g)$ be a smooth, connected, orientable, and closed Riemannian manifold and $M\subset N$ an open set whose boundary $\p M$ is a smooth strictly convex hyper-surface of $(N,g)$. There exists an open subset $\hat{M}$ of $N$, that contains the closure of $M,$ and whose boundary is a smooth, strictly convex hyper-surface of $N$. 

Moreover
\begin{equation}
    \label{eq:dist_agree}
d_{\hat{M}}(p,q)=d_{M}(p,q), \quad \text{ for all } p,q \in \bar M.
\end{equation}
\end{proposition}
\begin{proof}
Since $\p M$ is a smooth hyper-surface of $N$ there exists a smooth function $s\colon N \to \R$ and a neighborhood $U$ of $\p M$ such that
\[
|s(x)|=\text{dist}(x,\p M):=\inf\{d_N(x,z): \: z \in \p M\}, \quad \text{and} \quad  \|\text{grad}\; s(x)\|_g\equiv 1,
\]
for every $x \in U$. Moreover for each $x \in U$ there exists a unique $z \in \p M$ such that
$
d_N(x,z)=|s(x)|.
$
We choose the sign convention of $s$ such that $s(x)\geq 0$ for $x \in U \setminus M$. Then on $\p M$ the gradient of the function $s(\cdot)$ agrees with the outward pointing unit normal vector field of $\p M$. The existence of this function is explained for instance in \cite[Example 6.43]{lee2018introduction}.

By this construction, each $p\in U$ can be written uniquely as
\[
p=(z(p),s(p))\in \p M \times \R,
\]
where $z(p)$ is the closest point of $\p M$ to $p$. Thus on $U$ we write the Riemannian metric as a function of $(z,\eps)\in  \p M \times \R$ in the form
$
\mathrm{ds}^2+\tilde{g}(\eps,z),
$
where $\tilde{g}(\eps,z)$ is the first fundamental form of the smooth hyper-surface 
$
\Omega(\eps):=s^{-1}\{\eps\}.
$
By \cite[Proposition 8.18]{lee2018introduction} we can then write the second fundamental form of $\Omega(\eps)$
as a bi-linear form 
\[
\Pi_{(z,\eps)}(X,Y)=-\frac{1}{2}\frac{\p}{\p \eps}\tilde{g}_{\alpha\beta}(\eps,z) X^\alpha Y^\beta\in \R
\]
on $T\Omega(s)$. Thus the eigenvalues $\lambda_1(z,\eps), \ldots,  \lambda_{n-1}(z,\eps)$ of $\Pi_{(z,\eps)}$ are continuous functions of $(z,\eps)$ \cite[Appendix V, Section 4, Theorem 4A]{whitney1972complex}. Since $\Omega(0)$ coincides with $\p M$, which is strictly convex, we have that $\lambda_{\alpha}(z,0)<0$ for every $\alpha \in \{1,\ldots,n-1\}$. Thus there exists $\eps_0>0$ so that 
\[
\lambda_{\alpha}(z,\eps)<0, \quad \text{for every $\alpha \in \{1,\ldots,n-1\}$ and $|\eps|<\eps_0$}.
\]
Therefore, for small enough 
$\eps>0$,
we have that
\[
M(\eps):=s^{-1}(-\infty,\eps)\subset M \cup U
\]
is an open set of $N$ that contains $\bar M$, and whose boundary
$
\p M(\eps)=\Omega(\eps)
$
is a smooth strictly convex hyper-surface of $N$. We choose $\eps\in (0,\eps_0)$ and set $\hat{M}=M(\eps)$.

\medskip

Let $p,q \in \bar M$ and choose a distance minimizing unit speed geodesic $\gamma\colon [0,d_{\hat M}(p,q)] \to \hat M$ that connects these points. Now without loss of generality we may assume that $\gamma(\tilde{t}) \in U$ for some $\tilde{t} \in [0,d_{\hat M}(p,q)]$. If this is not true then the trace of $\gamma$ is contained in $M$ and we are done.

Since $U$ is open and $\gamma(\tilde t)\in U$ we can choose an interval $[a,b]\subset [0,d_{\hat M}(p,q)]$ such that $\gamma([a,b])\subset U$ and define a smooth function
\[
\tilde{s} \colon [a,b] \to \R, \quad \tilde{s}(t):= s(\gamma(t)).
\]
Since $p$ and $q$ are in $\bar M$ we may without loss of generality assume that 
$
\tilde{s}(a),\tilde{s}(b)\leq 0.
$

We  aim to verify that $\tilde s$ is always non-positive. To establish this we show that the maximum value $m\in \R$ of $\tilde{s}$ is attained at the endpoints of the domain interval. So suppose that $m=\tilde{s}(t_0)$ is attained in some interior point $t_0\in (a,b)$. As $t_0$ is a maximum point of $\tilde{s}$, laying in the interior of the domain interval, it must hold that $\dot{\tilde{s}}(t_0)=0$ and $\ddot{\tilde{s}}(t_0)\leq 0$. On the other hand since $\gamma$ is a geodesic, we have by Weingarten equation \cite[Theorem 8.13 (c)]{lee2018introduction} that
\begin{equation}
    \label{eq:der_test}
\dot{\tilde s}(t_0)=\langle \grad s(\gamma(t_0)), \dot{\gamma}(t_0)\rangle_g, \: \text{ and } \:  \ddot{\tilde s}(t_0)=\langle D_t\grad s(\gamma(t_0)), \dot{\gamma}(t_0)\rangle_g=-\Pi_{\gamma(t_0)}(\dot \gamma(t_0),\dot{\gamma}(t_0)).
\end{equation}
Here $D_t$ stands for the covariant differentiation along the curve $\gamma$. Therefore $\dot{\gamma}(t_0)$ is tangential to the strictly convex hyper-surface $\Omega(m)$ which implies that 
$
\Pi_{\gamma(t_0)}(\dot \gamma(t_0),\dot{\gamma}(t_0))<0.
$
This in conjunction with \eqref{eq:der_test} leads into a contradiction with $\ddot{\tilde{s}}(t_0)\leq 0$. We have verified that for all $p,q \in \bar M$ any distance minimizing geodesic in $\hat{M}$, between these points, is contained in $\bar M$. Therefore the equation \eqref{eq:dist_agree} is true.
\end{proof}

By Proposition \ref{prop:extension} we may assume that $M$ is contained in the interior of some compact, Riemannian manifold $(\hat M,g)$ with a smooth strictly convex boundary. Moreover the distance function of $\hat M$ restricts to the one of $M$. 
Thus for every $(p,v)\in S\hat M$ where $p$ is in $M$ we always have that
\begin{equation}
    \label{eq:order_of_critical_funcs}
\cut(p,v)\leq \hat \cut(p,v), \quad \text{ and } \quad \exit(p,v)<\hat \exit(p,v),
\end{equation}
where $\hat \cut$ and $\hat \exit$ are the cut distance and the exit time functions of $(\hat M, g)$ respectively. 
Motivated by this observation we define the cut locus of a point $p \in M$ as
\begin{equation}
    \label{eq:cut_locus_real_def}
\begin{split}
\locus(p):=&\{\gamma_{p,v}(\cut(p,v))\in M: \: v \in S_pM,   
\:\cut(p,v)=\hat \cut(p,v)\}.
\end{split}
\end{equation}
The following result summarizes the basic properties of these sets. 

\begin{proposition}
\label{Prop:cut_locus}
Let Riemannian manifold  $(M,g)$ be as in Theorem \ref{prop:dist_smoothness}. Let $p \in M$.
\begin{itemize}
\item

The cut locus $\locus(p)$ of the point $p$ is a closed set of measure zero.

\item

If $q \in \locus(p)$ and $\gamma$ is a unit speed distance minimizing geodesic of $M$ between $p$ and $q$ then at least one of the following holds:
\begin{enumerate}
\item There exists another distance minimizing geodesic from $p$ to $q$.
\item $q$ is the first conjugate point to $p$ along $\gamma$.
\end{enumerate}
\end{itemize}
\end{proposition}
\begin{proof}
\noindent
\begin{itemize}
\item 
The proof of the first claim is identical to the proof of \cite[Theorem 10.34 (a)]{lee2018introduction}, thus omitted here.      



\item
If  $q \in \locus(p)$ then by \eqref{eq:order_of_critical_funcs} there is $v \in S_pM$ such that
\[
q=\gamma_{p,v}(\cut(p,v)),  \quad \cut(p,v)=\hat\cut(p,v)\leq \exit(p,v)<\hat \exit(p,v),
\]
Thus Lemma \ref{lem:cut_time} and Proposition \ref{prop:extension} yield the second claim.
\end{itemize}
\end{proof}

The following result introduces an open and dense subset of $M$ where the distance function of an interior point is smooth.
\begin{lemma}
\label{lem:outside_cut_locus}
Let Riemannian manifold  $(M,g)$ be as in Theorem \ref{prop:dist_smoothness}. Let $p \in M$. The distance function
$
d(p,\cdot) \colon M \to \R, 
$
is smooth precisely in the open and dense set  $M\setminus (\{p\} \cup \locus(p))$.
\end{lemma}

\begin{proof}
The proof is identical to the proof of the analogous claim of \cite[Chapter 5, Section 9, Corollary 7]{petersen2006riemannian} thus omitted here.
%
%
\end{proof}

\begin{proposition}
\label{cor:smoothness_of_d}
Let Riemannian manifold  $(M,g)$ be as in Theorem \ref{prop:dist_smoothness}. Let $p \in M$ and $q \in M \setminus ( \{p\} \cup \locus(p))$. There exist neighborhoods $U\subset M$ of $p$ and $V\subset M$ of $q$ such that the distance function $d(\cdot,\cdot)$ is smooth in the product set $U \times V$.
\end{proposition}
\begin{proof}
Let $(N,g)$ be a closed extension of $(M,g)$ as in Proposition \ref{prop:extension}. We define a smooth map 
\[
F: (x,v) \in TN \mapsto (x,\exp_x(v)) \in N\times N.
\]
Then the differential of this map can be written as 
\[ 
\mathrm{D}F(x,v)
=
\begin{bmatrix}
\text{Id}								& 0 \\
*	 	& \mathrm{D}\exp_x(v)
\end{bmatrix}.
\]
Since $q$ is not in the cut locus of $p$ there is a $v_0 \in T_pN$ such that $\|v_0\|_g=d_M(p,q)$, $\exp_{p}(v_0)=q$ and $\mathrm{D}\exp_p(v_0)$ is not singular. Therefore $\det(\mathrm{D} F(p,v_0) )= \det (\mathrm{D} \exp_p(v_0))$ does not vanish. Thus the Inverse function theorem implies that there are neighborhoods $\tilde W\subset TN$ of $(p,v_0)$ and $ W\subset N \times N$ of $(p,q)$ such that the local inverse function of $F$,
\[
F^{-1}\colon W \to \tilde W, \quad F^{-1}(x,y) =(x, \exp_x^{-1}(y)),
\]
is a diffeomorphism.

Let $(\hat M,g)$ be the extension of $(M,g)$ as in Proposition \ref{prop:extension}. Since $q$ is not in the cut locus of $p$ we have 
$
\|F^{-1}(p,q)\|_g<\hat\cut\left(p,\frac{v_0}{\|v_0\|_g}\right).
$
Thus by the continuity of the cut distance function $\hat\cut\left(\cdot,\cdot\right)$ of $(\hat M,g)$ we can choose a neighborhood $ W_1\subset  W$ of $(p,q)$ such that 
\[
\|F^{-1}(x,y)\|_g<\hat\cut\left(x,\frac{F^{-1}(x,y)}{\|F^{-1}(x,y)\|_g}\right), \quad \text{for all } (x,y)\in W_1.
\]
This gives
$
d_{\hat M}(x,y)=\|F^{-1}(x,y)\|_g, \text{ for all } (x,y)\in W_1.
$

 Finally we choose disjoint neighborhoods $U\subset M$ of $p$ and $V\subset M$ of $q$ such that $U\times V$ is contained in $W_1$. Let $(x,y)\in U \times V$ then $\gamma(t):=\exp_{x}(tF^{-1}(x,y))$ for $t\in[0,1]$ is a geodesic of $\hat M$ that connects $x$ to $y$ having the length of $d_{\hat M}(x,y)$. Since both $x$ and $y$ are in $M$, we get from the proof of Proposition \ref{prop:extension} that $\gamma(t)$ is contained in $M$. This yields
\begin{equation}
    \label{eq:smoothness_of_d}
d_{M}(x,y)=\|F^{-1}(x,y)\|_g, \text{ for all } (x,y)\in U \times V.
\end{equation}
Since the sets $U$ and $V$ are disjoint we have that $F^{-1}$ does not vanish in $U\times V$. Hence equation \eqref{eq:smoothness_of_d} gives the smoothness of $d_M(\cdot,\cdot)$ on $U\times V$.
\end{proof}

Recall that we have assumed that $M$ is isometrically embedded in a closed Riemannian manifold $(N,g)$. Thus any geodesic starting from $M$ can be extended to the entire $\R$. Let $p \in N$. We define the \textit{conjugate distance function} $\con\colon S_pN \to \R \cup {\infty}$ by the formula:
\[
\con(p,v)=\inf\{t>0 : \: \gamma_{p,v}(t) \text{ is a conjugate point  to $p$}\}.
\]
As the infimum of the empty set is positive infinity we set $\con(p,v)=\infty$ in the case when the geodesic $ \gamma_{p,v}$ does not have any conjugate points to $p$. 
Since geodesics do not minimize the distance beyond the first conjugate point it holds that
\[
\cut(p,v)\leq \con(p,v), \quad \text{ if } (p,v) \in SM.
\]

The following result is well known, but we could not find its proof in the existing literature, so we provide one below. 
\begin{lemma}
\label{lem:conj_dist_func_is_cont}
Let $(N,g)$ be a closed Riemannian manifold and $p \in N$. The conjugate distance function is continuous on $S_pN$.
\end{lemma}
\begin{proof}
Let $v_i\in S_pN$ for $i \in \N$ converge to $v$.

We set
\[
C=\con(p,v), \quad B=\liminf_{i \to \infty}\con(p,v_i), \quad \hbox{ and } \quad A=\limsup_{i\to \infty}\con(p,v_i).
\]
It suffices to show that $A\leq C \leq B$. 

\medskip
We assume first that $C=\infty$. If $A<\infty$ we choose a sub-sequence $v_{i_k}$ of $v_i$ such that $\con(p,v_{i_k})$ converges to $A$. Then
$
\det(\mathrm D\exp_{p}(\con(p,v_{i_k})v_{i_k}))=0,
$
and the smoothness of the exponential map gives $\det(\mathrm D\exp_{p}(Av))=0$ yielding that $\exp_{p}(Av)$ is conjugate to $p$ along $\gamma_{p,v}$. This implies that $\con(p,v)\leq A$, which is impossible. By the same argument we see that $B=\infty$.

\medskip
Let $C<\infty$, and by the same limiting argument as above we get $C\leq B$.
%
Then we show that $A\leq C$.  Choose a sub-sequence $v_{i_k}$ such that $\con(p,v_{i_k})\to A$ as $k\to\infty$.  For the sake of contradiction we suppose that $A>C$. By the definition of the conjugate distance function we have that $q=\gamma_{p,v}(C)$ is the first conjugate to $p$ along $\gamma_{p,v}$. By \cite[Theorem 10.26]{lee2018introduction} for any  $\eps\in(0,A-C)$ there exists a piecewise smooth vector field $X$ on the geodesic segment $\gamma_{p,v}\colon {[0,C+\eps]} \to N$, that vanishes on $0$ and $C+\eps$ such that the index form of $\gamma_{p,v}$ over $X$ is strictly negative. That is
\begin{equation}
\label{eq:index_of_X}
I_{v}(X,X):=\int_{0}^{C+\eps}\langle \mathrm D_t X,  \mathrm D_t X\rangle_g+\langle R(\dot \gamma_{p,v},X)\dot \gamma_{p,v}, X \rangle_g \; \mathrm{d}t<0.
\end{equation}
Here we used the notation $\mathrm{D}_t$ for the covariant differentiation along $\gamma_{p,v}$. The capital $R$ stands for the Riemannian curvature tensor. 

We choose vectors $E_1,\ldots,E_n$ of $T_pN$ that form a basis of $T_pN$ and extend them on $\gamma_{p,v}(t)$ for $t \in [0,C+\eps]$ via the parallel transport.  Since parallel transport is an isomorphism the vector fields $\{E_1(t), \ldots,E_{n}(t)\}$ constitute a basis of $T_{ \gamma_{p,v}(t)}M$.  We write $X(t)=X^j(t)E_j(t)$.  Since $X$ is piecewise smooth it holds that the component functions $X^j(t)$ are piecewise smooth.  This lets us `extend' $X$ on $\gamma_{p,v_{i_k}}$ by the formula
\begin{equation}
\label{eq:vec_field_X}
X_k(t)=X^j(t)E_j^k(t),
\end{equation}
where the vector field $E_j^k(t)$ is the parallel transport of $E_j$ along $\gamma_{p,v_{i_k}}$. Thus $X_k$ is a piecewise smooth vector field on $\gamma_{p,v_{i_k}}$ that vanishes at $t=0$ and $t=C+\eps$. 

Since 
$
v_{i_k}\to v,   \hbox{ as $k \to \infty$}, 
$
it holds that 
\[
\gamma_{p,v_{i_k}}(t) \to \gamma_{p,v}(t), \quad \hbox{ and } \quad  E^k_j(t)\to E_j(t), \quad \hbox{uniformly in $t \in [0,C+\eps]$ as $k \to \infty$. }
\]
Therefore by \eqref{eq:index_of_X}, \eqref{eq:vec_field_X}, the continuity of the Levi-Civita connection, and the Riemannian curvature tensors we have
\[
I_{v_{i_k}}(X_k, X_k)
<0, \quad \hbox{ for large enough $k\in \N$.}
\] 
By \cite[Theorem 10.28]{lee2018introduction} there exists $s_k \in (0,C+\eps]$ so that $\gamma_{p,v_{i_k}}(0)$ and $\gamma_{p,v_{i_k}}(s_k)$ are conjugate points. Therefore we must have $s_k\geq \con(p,v_{i_k})$ and we arrive at a contradiction  
$
\con(p,v_{i_k}) < A.
$
This ends the proof.
\end{proof}

\subsection{The Hausdorff dimension of the cut locus}
We fix a point $p\in M$ for this sub-section. By Lemma \ref{lem:outside_cut_locus} we know that for any $p \in M$ the distance function $d(p,\cdot)$ is smooth in the open set $M\setminus(\locus(p)\cup\{p\})$. Moreover by Proposition \ref{cor:smoothness_of_d}  for each $q$ in this open set there are neighborhoods $U$ of $p$ and $V$ of $q$ such that the distance function is smooth in the product set $U\times V$. As we are interested in the inverse problem where we study distance function $d(p,\cdot)$ restricted on some open subset $\Gamma$ of the boundary, we do not know \textit{a priori} if this function is smooth on $\Gamma$. In particular we do not know the size of the set $\locus(p)\cap \p M$ yet. In this sub-section we show that the set $\p M\setminus\locus(p)$, where $d(p,\cdot)$ is smooth, is always an open and dense subset of $\p M$.

Proposition \ref{Prop:cut_locus} yields that $\locus(p)$ can be written as a disjoint union of 
\begin{itemize}
\item \textit{Conjugate cut points:} 
\[
Q(p):=\{\gamma_{p,v}(t)\in M : v \in S_pM, \: t= \cut(p,v)= \con(p,v)\}\subset \locus(p),
\] 
that are those points $q \in \locus(p)$ such that there exits a distance minimizing geodesic from $p$ to $q$ along which these points are conjugate to each other.  By Proposition \ref{Prop:cut_locus} and Lemma  \ref{lem:conj_dist_func_is_cont} the set $Q(p)$ is closed in $M$. 
\item \textit{Typical cut points:} $T(p)\subset (cut(p)\setminus Q(p))$ that can be connected to $p$ with exactly two distance minimizing geodesics.
\item \textit{A-typical cut points:} $L(p) \subset (cut(p)\setminus Q(p))$ that can be connected to $p$ with more than two distance minimizing geodesics. Thus an a-typical cut point is both non-conjugate and non-typical.
\end{itemize}

It was proven in \cite{itoh1998dimension} that the Hausdorff dimension of the cut locus on a closed Riemannian manifold $(N,g)$ is locally an integer that does not exceed $\dim N -1$. Moreover $T(p)$ is a smooth  hyper-surface of $N$ and the Hausdorff dimension of $Q(p)\cup L(p)$ does not exceed $\dim N-2$. In this paper we will extended these results for manifolds with strictly convex boundary. The main result of this section is as follows:

\begin{theorem}
\label{thm:cut_locus}
Let $(M,g)$ be a smooth, compact, connected, and oriented Riemannian manifold of dimension $n\in \N, \: n\geq 2$ with smooth and strictly convex boundary. If $p \in M$ then
\begin{enumerate}
\item \label{C1}
The set $T(p)$ of typical cut points is a smooth hyper-surface of $M$ that is transverse to $\p M$. 
\item \label{C2}
The Hausdorff dimension of $Q(p)\cup L(p)$ does not exceed $n-2$. 
\item \label{C3} The Hausdorff dimension of $\locus(p)$ does not exceed $n -1$.  
\item \label{C4}
The set
$
\p M \setminus \locus(p)
$
is open and dense in $\p M$. 
\end{enumerate}
\end{theorem}
For the readers who want to learn more about Hausdorff measure and dimension we suggest to have look at  \cite{burago2001course, mattila1999geometry}. Some basic properties of the Hausdorff dimension $\dim_{\mathcal{H}}$ 
are collected in the following lemma.

\begin{lemma}
\label{lem:Haus_dim_prop}
Basic properties of the Hausdorff dimension are:
\begin{itemize}
    \item If $X$ is a metric space and $A\subset X$ then $\dim_{\mathcal{H}}(A)\leq \dim_{\mathcal{H}}(X)$.
    \item If $X$ is a metric space and $\mathcal{X}$ is a countable cover of $X$ then $\dim_{\mathcal{H}}(X)=\sup_{A \in \mathcal{X}}\dim_{\mathcal{H}}(A)$.
    \item If $X,Y$ are metric spaces and $f\colon X \to Y$ is a bi-Lipschitz map then $\dim_{\mathcal{H}}(A)=\dim_{\mathcal{H}}(f(A))$ for any $A \subset X$.
    \item If $U \subset M$ is open and $M$ is a Riemannian manifold of dimension $n\in \N$ then $\dim_{\mathcal{H}}(U)=n$. 
\end{itemize}
\end{lemma}

From here onwards we follow the steps of \cite{itoh1998dimension, ozols1974cut} and develop machinery needed for the proof of Theorem \ref{thm:cut_locus}. We recall that we have isometrically embedded $M$ into the closed Riemannian manifold $(N,g)$. The maximal subset $M_p \subset T_pM$ where the exponential map $\exp_p\colon M_p \to M$ of $(M,g)$ is well defined was given in \eqref{eq:M_p_set} as
\[
M_p=\left\{v\in T_pM: \: v=0 \hbox{ or } \|v\|_g\leq \exit\left(p,\frac{v}{\|v\|_{g}}\right)\right\}.
\]
Thus the exponential function $\exp_p\colon T_pN \to N$ of $N$ agrees with that of $M$ in $M_p$.

Let $v_0 \in M_p$ be such that the exponential map $\exp_p$ is not singular at $v_0$. The Inverse function theorem yields that there are neighborhoods $ U \subset T_pN$ of $v_0$ and $\tilde V \subset N$ of $x_0=\exp_p(v_0) \in M$ such that $\exp_p\colon  U \to \tilde V$ is a diffeomorphism. We want to emphasize that even when $v_0\in M_p$ the set $U\subset T_pN$ does not need to be contained in $M_p$.  If we equate $T_pN$ and $T_{v_0}(T_pN)$ we note that the formula
\[
Y(x)=\mathrm{D} \exp_p\bigg|_{ \exp_{p}^{-1}(x)} \exp_{p}^{-1}(x)
\]
defines a smooth vector field on $\tilde V$ that satisfies the following properties
\begin{equation}
\label{eq:props_of_Ys}
Y(x)=\dot{\gamma}_{p, \exp_{p}^{-1}(x)}(1), \quad \hbox{ and } \quad \|Y(x)\|_g=\| \exp_{p}^{-1}(x)\|_g.
\end{equation}
The vector field $Y$ is called a \textit{distance vector field} related to $p$ and $ U$. Let $x \in \tilde V$ and $X\in T_{x} N$.  It holds by a similar proof to \cite[Lemma 2.2.]{ozols1974cut} that 
\begin{equation}
\label{eq:der_of_Y}
X\|Y(x)\|_g=\frac{\langle X,Y(x)  \rangle_g}{\|Y(x)\|_g}.
\end{equation}

In what follows we will always consider $\locus(p)$ as defined for $(M,g)$ in equation \eqref{eq:cut_locus_real_def}.  Let $q \in \locus(p)\setminus Q(p)$ and $\lambda=d(p,q)$.  By Proposition \ref{Prop:cut_locus} it holds that there are at least two $M$-distance minimizing geodesics from $p$ to $q$. Thus the set
\[
E_{p,q}:=\exp_p^{-1}\{q\}\cap S_\lambda M,  \quad \hbox{ where }S_ \lambda M=\{w \in T_pM: \: \|w\|_g=\lambda\},
\]
contains at least two points. 



By the compactness of $S_\lambda M$ and the continuity of the exit time function 
on the non-trapping part of $SM$
we get from the assumption that $q\notin Q(p)$ that the set $E_{p,q}$ is finite. We write 
\[
E_{p,q}=\{w_i : \: i \in \{1,\ldots, k_p(q)\}\},
\]
where $k_p(q) \in \N$ is the number of distance minimizing geodesics from $p$ to $q$.  Since the set $Q(p)$ is closed in $M$, the complement $\locus(p)\setminus Q(p)$ is relatively open in $\locus(p)$, and there exists an open neighborhood $W \subset M$ of $q$ such that
$
Q(p)\cap W=\emptyset.
$  
Thus by the previous discussion for any $x \in \locus(p)\cap W$ there are only $k_p(x) \in \N$ many distance minimizing geodesics connecting $p$ to $x$. 
Moreover, as the following lemma shows, $q$ is a local maximum of the function $k_p$ defined on $\locus(p) \cap W.$
This statement is an adaptation of the analogous result given in \cite{ozols1974cut}. 
\begin{lemma}
\label{lem:in_eq_for_ks}
Let $(M,g)$ be a Riemannian manifold as in Theorem \ref{thm:cut_locus}. Let $p \in M$ and $q \in \locus(p)\setminus Q(p)$. Let the closed manifold $(N,g)$ be as in Proposition \ref{prop:extension}. Then there is a neighborhood $V$ of $q$ in $M$ such that 
\begin{equation}
\label{eq:k(z)<k(q)}
k_p(x)\leq k_p(q), \quad \hbox{ for every } x \in \locus(p)\cap V.
\end{equation}
\end{lemma}
\begin{proof}
Since the set $E_{p,q}$ is finite we can choose disjoint neighborhoods $U_i \subset  T_pN$ for each $w_i \in E_{p,q}$, so that for each $i \in \{1,\ldots,k_p(q)\}$ the map
$
\exp_p\colon U_i \to \tilde V 
$
is a diffeomorphism on some open set $\tilde V\subset N$ that contains $q$.  We want to show that there is a neighborhood $V \subset M$ of $q$ such that for every $x \in V$ and for any $M$-distance minimizing unit speed geodesic $\gamma$ from $p$ to $x$ there is $i \in \{1, \ldots, k_p(q)\}$ such that
\[
\gamma(t)=\exp_p\left(t\frac{X}{\|X\|_g}\right), \quad \hbox{ for some } X \in M_p\cap U_i.
\]
Clearly this implies the inequality \eqref{eq:k(z)<k(q)}.

If the former is not true then there exist a sequence $q_k \in M$ that converges to  $q$ and  $X_k \in M_p$ so that for each $k \in\N$ we have
\begin{itemize}
\item 
$\exp_p(X_k)=q_k$
\item 
$\|X_k\|_g=d_M(p,q_k)\leq \exit(p,\frac{X_k}{\|X_k\|_g})$
\item
$\exp_p\left(t\frac{X_k}{\|X_k\|_g}\right)$ for $t \in [0,d_M(p,q_k)]$ is a unit speed distance minimizing geodesic from $p$ to $q_k$. 
\item
$X_k \notin  U_1\cup\ldots \cup U_{k_p(q)}$.
\end{itemize}
These imply that
\[
\lim_{k \to \infty } \|X_k\|_g=\lim_{k \to \infty} d_M(p,q_k)=d_M(p,q).
\]
Moreover, the sequence $X_k \in T_pM$ is contained in some compact subset $K$ of $T_pM$. After passing to a sub-sequence we may assume $X_k \to X \in T_pM$ and the continuity of the exit time function 
on the non-trapping part of $SM$
gives $\|X\|_g\leq \exit(p,\frac{X}{\|X\|_g})$. Thus $X\in M_p$ and by the continuity of $\exp_p$ we get
$
\exp_p(X)=q, \hbox{ and } \|X\|_g=d_M(p,q).   
$

Therefore $t\mapsto \exp_p\left(t\frac{X}{\|X\|_g}\right)$ is a $M$-distance minimizing geodesic from $p$ to $q$ and $X$ must coincide with $w_i$ for some $i \in \{1,\ldots,k_p(q)\}$. Therefore $X_k \in U_i$ for large enough $k \in \N$. This contradicts the choice of $X_k$, and possibly after choosing a smaller $\tilde V$, we can set $V=M \cap \tilde V$.  
\end{proof}

Suppose now that $q \in T(p)$ is a typical cut point, and $V\subset M$ is a neighborhood of $q$ as in Lemma \ref{lem:in_eq_for_ks}.  Then by \eqref{eq:k(z)<k(q)} it holds that 
\begin{equation}
\label{eq:lous_cap_V}
\locus(p) \cap V \subset T(p),
\end{equation}
and $E_{p,q}=\{w_1,w_2\}\subset T_{p}N$ are the directions that give the two distance minimizing geodesics $\exp_{p}(tw_i), \: t\in [0,1]$ from $p$ to $q$.
Let $U_1,U_2\subset T_pN$ be the neighborhoods of $w_1$ and $w_2$ and $\tilde{V}\subset N$ a neighborhood of $q$ as in the proof of Lemma \ref{lem:in_eq_for_ks}. 
Finally we consider the distance vector fields $Y_1,Y_2$ related to $p$ and $ U_1$ and $ U_2$.  Since these vector fields do not vanish on $\tilde V$ the function
\[
\rho\colon \tilde V \to \R, \quad \rho(x)=\|Y_1(x)\|_g-\|Y_2(x)\|_g,
\]
is smooth. The following result is an adaptation of \cite[Propositions 2.3 \& 2.4]{ozols1974cut}.

\begin{lemma}
\label{prop:hyper_surface}
Let Riemannian manifold $(M,g)$ be as in Theorem \ref{thm:cut_locus} and $p \in M$. Let $q \in T(p)\subset M$ and define the closed manifold $N$ as in Proposition \ref{prop:extension}. Let the neighborhood $\tilde V \subset N$ of $q$ and function $\rho\colon \tilde V \to \R$ be as above. Then possibly after choosing a small enough $\tilde V$ we have
\begin{equation}
\label{eq:rho_level_set}
\rho^{-1}\{0\}\cap M=\tilde V\cap \locus(p).
\end{equation}  
Moreover, the set $\rho^{-1}\{0\}$ is a smooth hyper-surface of $N$ whose tangent bundle is given by the orthogonal complement of the vector field $Y_1-Y_2$. 
\end{lemma}
\begin{proof}
We prove first the equation \eqref{eq:rho_level_set}. 
\begin{itemize}
\item
Let $x \in \rho^{-1}\{0\}\cap M$.  By the proof of Lemma \ref{lem:in_eq_for_ks} we can assume that a $M$-distance minimizing unit speed geodesic from $p$ to $x$ is given by
$
\exp_p\left(tX_1\right), \: t\in [0,1], \: \hbox{ for some } X_1 \in M_p\cap U_1.
$  
Also $x=\exp_{p}(X_2)$ for some $X_2 \in  U_2$, but we do not know \textit{a priori} if $\exp_{p}(tX_2)\in M$ for all $t\in [0,1]$ or equivalently if $X_2 \in M_p$. However, by the definition of the distance vector fields
and the assumption $x \in \rho^{-1}\{0\}$ we have that
\begin{equation}
    \label{eq:length_of_X}
d_M(p,x)=\|X_1\|_g=\|Y_1(x)\|_g=\|Y_2(x)\|_g=\|X_2\|_g.
\end{equation}

Let $\hat{M}$ be as in Proposition \ref{prop:extension}. Thus we can assume that $\tilde V \subset \hat M$. Since $q \in T(p)$ there is $w_2 \in U_2$ so that $\exp_p(w_2)=q$ and
$
\exp_p(tw_2)\in M \subset \hat M,
$
for every $t\in[0,1]$. Since $q$ is an interior point of $\hat M$ we can again choose smaller $\tilde{V}$ so that 
$
\exp_p(tX) \in \hat M,
$
for every $t\in[0,1]$ and $X \in U_i$, for $i \in \{1,2\}$. Since $x \in M$  and $\exp_p(tX_2)$ is a geodesic of $\hat M$ that connects $p$ to $x$ having the length of $\|X_2\|_g$, the equation \eqref{eq:length_of_X} and Proposition \ref{prop:extension} imply that $\exp_p(tX_2) \in M$ for every $t\in[0,1]$. Therefore equation \eqref{eq:length_of_X} gives $x \in \tilde V\cap \locus(p)$.  


\item
Let $x \in \tilde V\cap \locus(p) \subset T(p)$. Thus there are exactly two distance minimizing geodesics of $M$ from $p$ to $x$.  Since $x \in \tilde V$, it holds by the proof of Lemma \ref{lem:in_eq_for_ks} that one of these geodesics has the initial velocity in $U_1$ and the other in $U_2$. Therefore $\rho(x)$ is zero by the definition of the distance vector fields.
\end{itemize}

Then we prove that the set $\rho^{-1}\{0\}$ is a smooth hyper-surface whose tangent bundle is orthogonal to the vector field $Y_1-Y_2$. By \eqref{eq:der_of_Y} we get
\begin{equation}
    \label{eq:diff_of_rho}
X\rho(x)=\frac{\langle X,Y_1(x)  \rangle_g}{\|Y_1(x)\|_g}-\frac{\langle X,Y_2(x) \rangle_g}{\|Y_2(x)\|_g}=\frac{\langle X,Y_1(x)-Y_2(x)  \rangle_g}{\|Y_1(x)\|_g}, \quad \hbox{ for every } x \in \rho^{-1}\{0\}.
\end{equation}

Moreover the vector field $Y_1-Y_2$ does not vanish on $\tilde V$, since the geodesics related to these two vector fields are different. 
This implies that the differential of the map $\rho$ does not vanish in $\tilde V$. Thus the set $\rho^{-1}\{0\}$ is a smooth hyper-surface of $N$, and by \eqref{eq:diff_of_rho} its tangent bundle is given by those vectors that are orthogonal to  $Y_1-Y_2$.
\end{proof}

\medskip

Now we consider the set of conjugate cut points $Q(p)$. First we define a function
\[
\delta\colon S_pN \to \{0,1,\ldots,n-1\}, \quad \delta(v) \text{ is the dimension of the kernel of }  \mathrm{D}\exp_p \text{ at } \con(p,v)v.
\]
If $\cut(p,v)=\infty$ we set $\delta(v)=0$. 

\begin{lemma}
\label{lem:delta_func_is_locally_constant}
Let $(N,g)$ be a closed Riemannian manifold. Let $p \in N$ and $v_0 \in S_pN$ be such that $\delta(v_0)=1$. There exists a neighborhood $U \subset S_pN$ of $v_0$ such that the $\delta(\cdot)$ is the constant function one in $U$. 
\end{lemma}

Before proving this lemma we recall one auxiliary result from linear algebra.
\begin{lemma}
\label{lem:eigenvalues_vs_index}
Let $L\colon \R^n \to \R^n$ be a self-adjoint bijective linear operator. Then the index $i(L)$ of $L$, the dimension of the largest vector subspace of $\R^n$ where $L$ is negative definite, equals the amount of the negative eigenvalues of the operator $L$ counted up to a multiplicity. 
\end{lemma}
\begin{proof}
The proof follows directly from the spectral theorem and thus we omit it here.
%
\end{proof}
\begin{proof}[Proof of Lemma \ref{lem:delta_func_is_locally_constant}]
In this proof we adopt the definitions and results of \cite[Chapter 11]{carmo1992riemannian} appearing in the proof of the Morse index theorem. For each $v\in S_pM$ and $t>0$ we use the notation $\mathcal{V}(t,v)$ for the vector space of all piecewise smooth vector fields that are normal to the geodesic $\gamma_{p,v}$ in the interval $[0,t]$ and vanish at the endpoints. Then we define the function
$
i\colon [0,\infty) \times S_pN \to \N, 
$
to be the index of the symmetric bilinear form
\[
I_{t,v}(X,Y):=\int_{0}^{t}\langle \mathrm D_s X,  \mathrm D_s Y\rangle_g+\langle R(\dot \gamma_{p,v},X)\dot \gamma_{p,v}, Y \rangle_g \; \mathrm{d}s, \quad X,Y \in \mathcal{V}(t,v).
\]

Hence,
\[
\left\{\begin{array}{ll}
   \delta(v)>i(t,v)=0,  &  t\leq \con(p,v)
   \\
   \delta(v)=i(t,v), & t\in (\con(p,v), \eps(v))
   \\
   \delta(v)< i(t,v), & t > \eps(v)
\end{array}\right.
\]
where $\eps(v)>0$ depends on $v \in S_pN$. Moreover, no $\gamma_{p,v}(t)$, for $t \in (\con(p,v), \eps(v))$ is conjugate to $p$ along $\gamma_{p,v}$. We choose $t \in (\con(p,v_0), \eps(v_0))$. Thus Lemma \ref{lem:conj_dist_func_is_cont}, gives $\delta(v)\leq i(t,v)$ for $v \in S_pN$ close enough to $v_0$. Our aim is to find a neighborhood $U\subset S_pN$ of $v_0$ for which
\begin{equation}
    \label{eq:morse_in_eq}
1\leq \delta(v)\leq i(t,v)= i(t,v_0)= \delta(v_0)=1, \quad  \text{for $v \in U$}.
\end{equation}
Clearly this gives the result of the lemma.

The space $\mathcal{V}(t,v)$ can be written as a direct sum of two of its vector sub-spaces $\mathcal{V}_+(t,v)$ and $\mathcal{V}_-(t,v)$, defined so that index form $I_{t,v}$ is positive definite on $\mathcal{V}_+(t,v)$ and the space $\mathcal{V}_-(t,v)$ is finite dimensional. Moreover, these vector spaces are $I_{t,v}$ orthogonal. Thus the index of $I_{t,v}$ coincides with the index of its restriction on $\mathcal{V}_-(t,v)$. Since the dimension of $\mathcal{V}_-(t,v)$ is independent of a direction $v \in S_pN$, that is close to $v_0$, we can identify all the spaces $\mathcal{V}_-(t,v)$ with $\mathcal{V}_-(t,v_0)$ and consider the bilinear forms $I_{t,v}$ as a family of operators on the finite dimensional vector space $\mathcal{V}_-(t,v_0)$, depending continuously on the parameter $v \in S_pN$. 

For each $v \in S_pN$, we consider the linear operator $L_{t,v}\colon\mathcal{V}_-(t,v) \to \mathcal{V}_-(t,v) $, corresponding to the bilinear form $I_{t,v}$. Since $\gamma_{p,v_0}(t)$ is not a conjugate point to $p$ along $\gamma_{p,v_0}$, zero is not an eigenvalue of the linear operator $L_{t,v_0}$. Thus Lemma \ref{lem:eigenvalues_vs_index} implies that the operator $L_{t,v_0}$ has $i(t,v_0)$ negative eigenvalues. Since the eigenvalues of the operator $L_{t,v}$ depend continuously on the initial direction $v \in S_pN$, that are near $v_0$, we can find a neighborhood $U\subset S_pN$ of $v_0$ such that the linear operator $L_{t,v}$ is invertible and has $i(t,v_0)$ negative eigenvalues for every $v \in U$. Hence, by Lemma \ref{lem:eigenvalues_vs_index} we have again that $i(t,v)=i(t,v_0)$ for $v \in U$. We have verified the equation \eqref{eq:morse_in_eq}.
\end{proof}

\begin{lemma}
\label{lem:smoothness_of_con}
Let $(N,g)$ be a closed Riemannian manifold, $p \in N$ and suppose that $\delta$ is constant in some open set $U \subset S_pM$. Then $\con(p,\cdot)$ is smooth in $U$. 
\end{lemma}
\begin{proof}
If $\delta$ is zero in $U$ then $\con(p,\cdot)$ is infinite and we are done. So we suppose that $\delta$ equals to $k\in \{1,\ldots,n-1\}$ in $U$ and get $\con(p,v)<\infty$ for every $v \in U$. 

Let $\xi_1,\ldots,\xi_{n-1}$ be a base of $T_{v_0}S_pM$ and use the formula
\[
J_{v_0,\beta}(t)=\mathrm{D}\exp_{p}\bigg|_{tv_0} t\xi_\beta, \quad \text{ for } \beta \in \{1,\ldots,n-1\},
\]
from \cite[Proposition 10.10]{lee2018introduction}, to define  $(n-1)$-Jacobi fields $J_{v_0,1}(t),\ldots, J_{v_0,n-1}(t)$ along the geodesic $\gamma_{p,v_0}$. They span the vector space of all Jacobi fields along
$\gamma_{p,v_0}(t)$
that 
vanish
at $t=0$ and are normal to 
$\dot{\gamma}_{p,v_0}(t)$.
As $\delta(v_0)=k$ we may assume that $J_{v_0,\beta}(\con(p,v_0))=0$ for $\beta\in \{1,\ldots,k\}$, implying $D_tJ_{v_0,\beta}(\con(p,v_0))\neq 0$ and $J_{v_0,\alpha}(\con(p,v_0))\neq 0$ for $\beta\in \{1,\ldots,k\}$ and $\alpha\in \{k+1,\ldots,n-1\}$. Moreover, the vectors 
\begin{equation}
\label{eq:row_vectors}
D_tJ_{v_0,1}(\con(p,v_0)),\ldots, D_tJ_{v_0,k}(\con(p,v_0)), J_{v_0,k+1}(\con(p,v_0)), \ldots J_{v_0,n-1}(\con(p,v_0))
\end{equation}
are linearly independent due to \cite[Proposition 2.5.8 (ii)]{klingenberg} and the properties of the geodesic flow on $SN$ presented in \cite[Lemma 1.40]{paternain2012geodesic}.

Since Jacobi fields are solutions of the second order ODE they depend smoothly on the coefficients of the respective equation. In particular, after choosing smaller $U$ if necessary, we can construct the family of Jacobi fields $J_{v,1}(t),\ldots, J_{v,n-1}(t)$ along the geodesic $\gamma_{p,v}(t)$ that depend smoothly on $v \in U$, and span the vector space of all Jacobi fields along $\gamma_{p,v}(t)$ that vanishes at $t=0$ and are normal to $\dot{\gamma}_{p,v}(t)$.
Therefore the function  
\[
f\colon U\times [0,\con(p,v_0)+1]\to \R, \quad  f(v,t)=\det (J_{v,1}(t),\ldots, J_{v,n-1}(t)),
\]
is smooth and vanishes at $(v,t)$ if and only if $\gamma_{p,v}(t)$ is conjugate to $p$. 

We choose a parallel frame $E_1(t),\ldots,E_{n-1}(t)$ along $\gamma_{p,v_0}(t)$ that is orthogonal to $\dot{\gamma}_{p,v_0}(t)$. With respect to this frame we write
\[
J_{v_0,\beta}(t)=j_{\beta}^\alpha(t)E_{\alpha}(t), \quad \hbox{ for } \alpha, \beta \in {1,\ldots,n-1},
\]
for some some smooth functions $j_\beta^\alpha(t)$. From here we get
\[
\frac{\p^{j}}{\p t^{j}}f(v_0,\con(p,v_0))=\frac{\p^{j}}{\p t^{j}}\left(\sum_{\sigma} \text{sign}(\sigma) j_1^{\sigma(1)}(t)j_2^{\sigma(2)}(t)\cdots j_{n-1}^{\sigma(n-1)}(t) \right)\bigg|_{t=\con(p,v_0)}=0,  
\]
$\hbox{for every } j\in \{0,\ldots,k-1\}$. Above, $\sigma$ is a permutation of the set $\{1,\ldots,n-1\}$. Moreover the covariant derivative of $J_{v_0,\beta}$ along $\gamma_{p,v_0}$ is written as
$
D_tJ_{v_0,\beta}(t)= \left(\frac{\mathrm{d}}{\mathrm{d}t}j_{\beta}^\alpha(t)\right)E_{\alpha}(t).
$

If $A$ is the square matrix whose column vectors are given in the formula \eqref{eq:row_vectors} we have 
\[
\begin{split}
&\frac{\p^{k}}{\p t^{k}}f(v_0,\con(p,v_0))
= 
\pm k!\det(A)\neq 0.
\end{split}
\]
Since $\delta(\cdot)$ is constant $k$ in the set $U$ we have that $\frac{\p^{k-1}}{\p t^{k-1}}f(v,\con(p,v))=0$ for every $v \in U$. Therefore the Implicit function theorem gives that the conjugate distance function is smooth in some neighborhood $V\subset U$ of $v_0$. Since $v_0\in U$ was chosen arbitrarily the claim follows.
\end{proof}

Let $v_0\in S_pN$ be such that $\con(p,v_0)<\infty$. Then by lemmas \ref{lem:cut_dis_is_cont} and \ref{lem:conj_dist_func_is_cont} the function
$
e_q(v)=\exp_p(\con(p,v)v)\in N
$
is well defined and continuous on some neighborhood $U\subset S_pN$ of $v_0$. Moreover, we have that
\begin{equation}
    \label{eq:conjugate_locus}
Q(p)=\{e_q(v)\in M: \: \cut(p,v)=\con(p,v)\}.
\end{equation}
The following result is an adaptation of \cite[Lemma 2]{itoh1998dimension}.
\begin{proposition}
 \label{prop:dim_of_Q(p)}
Let Riemannian manifold $(M,g)$ be as in Theorem \ref{thm:cut_locus} and $p \in M$. The Hausdorff dimension of $Q(p)$ does not exceed $n-2$. 
\end{proposition}
\begin{proof}
By \eqref{eq:conjugate_locus} we can write the conjugate cut locus $Q(p)$ as a disjoint union of the sets
\[
A_1=\{e_q(v)\in M: \: \cut(p,v)=\con(p,v), \: \delta(v)=1 \}
\]
and
\[
A_2=\{e_q(v)\in M: \: \cut(p,v)=\con(p,v), \: \delta(v)\geq 2 \}.    
\]
To prove the claim of this proposition it suffices to show that
\begin{equation}
    \label{eq:A_1_inclusion}
A_1\subset \{e_q(v)\in N: \: \dim \left(\mathrm{D}e_q(T_{v}S_pM)\right)\leq n-2 \},
\end{equation}
since clearly we have that
\[
A_2\subset \{\exp_{p}(w)\in N: \: w\in T_pN, \: \dim \left(\mathrm{D}\exp_{p}(T_w(T_pN))\right)\leq n-2\},
\]
and therefore by the generalization of the classical Sard's theorem \cite{sard1965hausdorff} the Hausdorff dimension of $Q(p)=A_1\cup A_2$ is at most $n-2$. 

\medskip

We choose $v_0 \in S_pN$ such that $e_q(v_0)\in A_1$. 
By the properties of the Jacobi fields normal to $\gamma_{p,v_0}$,
we can identify the kernel $\mathrm{D}\exp_{p}(\con(p,v_0)v_0)$ with some vector sub-space of 
$
T_{v_0}S_pM
.$

\medskip

Since $\dim T_{v_0}S_pM=n-1$ we can verify the inclusion \eqref{eq:A_1_inclusion} if we show that 
\begin{equation}
\label{eq:inclusion_of_kernels}
\ker \mathrm{D}\exp_{p}(\con(p,v_0)v_0)\subset \ker \mathrm{D}e_q(v_0).
\end{equation}
Since $\delta(v_0)=1$ we get by lemmas \ref{lem:delta_func_is_locally_constant} and \ref{lem:smoothness_of_con} that there exists a neighborhood $U\subset S_pN$ of $v_0$ where the  conjugate distance $\con(p,\cdot)$ and the map $e_q(v)=\exp_{p}(\con(p,v)v)$ are smooth. Let $\xi \in T_{v_0}S_pM$ be in the kernel of the differential of the exponential map. Then by the chain and Leibniz rules we get
\[
\begin{split}
\mathrm{D}e_q(v_0)\xi=
\dot{\gamma}_{p,v_0}(\con(p,v_0))\mathrm{D}\con(p,v_0)\xi.
\end{split}
\]
Therefore $\mathrm{D}e_q(v_0)\xi=0$ if and only if $\mathrm{D}\con(p,v_0)\xi=0$. So we suppose that $\mathrm{D}e_q(v_0)\xi\neq 0$. 

By the Rank theorem \cite[Theorem 4.12]{lee2013smooth} we get that the subset of $T_pN$, near $\con(p,v_0)v_0$, where $\mathrm{D}\exp_{p}$ vanishes 
is diffeomorphic to a smooth sub-bundle of $TU\subset TS_pN$. Then we use the existence of the ODE theorem to choose a smooth curve $v(\cdot)\colon (-1,1) \to U\subset  S_pN$ such that $v(0)=v_0$, $\dot{v}(0)=\xi$ and $\dot{v}(t)\in \ker \mathrm{D}\exp_{p}(\con(p,v(t))v(t))$ for every $t \in (-1,1)$. Thus $c(t):=e_{q}(v(t))$ is a smooth curve in $\hat M$ that satisfies
\begin{equation}
\label{eq:vel_of_c}
\dot{c}(t)=\dot{\gamma}_{p,v(t)}(\con(p,v(t)))\frac{d}{d t}(\con(p,v(t))).
\end{equation}
Since $\frac{d}{d t}(\con(p,v(t_0)))=\mathrm{D}\con(p,v_0)\xi$ we can assume that $\frac{d}{d t}(\con(p,v(t)))>0$ on some interval $(-\eps,\eps)$ for $0<\eps<1$.
Thus by equation \eqref{eq:vel_of_c} and the Fundamental theorem of calculus we get that the length of $c(t)$ on $[-\eps,0]$ is 
$
\con(p,v_0)-\con(p,v(-\eps)).
$
Below we denote the length of $c(t)$ as $\mathcal{L}(c)$.
From here by the assumption $\con(p,v_0)=\cut(p,v_0)$ and the triangle inequality we get
\[
\begin{split}
\con(p,v(-\eps)) \geq &d_{\hat M}(p,e_q(v(-\eps)))
\geq 
d_{\hat M}(p,e_q(v_0))-\mathcal{L}(c)
\geq \con(p,v(-\eps)),
\end{split}
\]
and the inequality above must hold as an equality. 
Therefore
\[
d_{\hat M}(p,e_q(v(-\eps)))+\mathcal{L}(c)=  d_{\hat M}(p,e_q(v_0)),
\]
and the curve $c(\cdot)\colon [-\eps, 0]\to \hat M$ is part of some distance minimizing geodesic $\gamma$ of ${\hat M}$ from $p$ to $e_q(v(0))$ that contains $e_q(v(-\eps))$. 
Thus we have after some reparametrization $t=t(s)$ that
\[
\gamma(s)=e_q(v(t(s)))=c(t(s))
\]
for every $t(s)\in (-\eps,0)$. By \eqref{eq:vel_of_c} we get that $\gamma$ is a parallel to $\gamma_{p,v(t)}$ for every $t \in (-\eps,0)$. This is not possible unless the geodesics $\gamma_{p,v(t)}$ are all the same for every $t \in (-\eps,0)$. Hence $v(t)$ and $c(t)$ are constant curves. This leads to a contradiction.
The inclusion \eqref{eq:inclusion_of_kernels} is confirmed and the proof is complete.
\end{proof}

We are ready to present the proof of Theorem  \ref{thm:cut_locus}.

\begin{proof}[Proof of Theorem \ref{thm:cut_locus}]
Let $p \in M$. In this proof we combine the observations made earlier in this section. The proofs of the four sub-claims are given below.
\begin{itemize}
\item[\eqref{C1}] By Lemma \ref{prop:hyper_surface} we know that $T(p)$ is a smooth hyper-surface of $M$ whose tangent space is normal to the vector field $\nu(q)=Y_1(q)-Y_2(q)$ for $q \in T(p)$. Since $Y_1(q)\neq Y_2(q)$ and $\|Y_1(q)\|_g=\|Y_2(q)\|_g$ we get from Cauchy-Schwarz inequality that
\[
\langle Y_1(q), \nu(q)\rangle
>0
\quad \text{ and } \quad 
\langle Y_2(q), \nu(q)\rangle
\rangle <0.
\]
Thus $Y_1(q)$ and $Y_2(q)$ hit $T(p)$ from different sides. If $q \in T(p)\cap \p M$ and these surfaces are tangential to each other at $q$ we arrive in a contradiction: Since $\nu(q)$ is normal to both $T(T(p))$ and $T\p M$ we can without loss of generality assume that $Y_2(q)$ is inward pointing at $q$. This is not possible since the geodesic related to $Y_2(q)$, that connects $p$ to $q$, is contained in $M$. Thus by equation \eqref{eq:props_of_Ys} $Y_2(q)$ is also outward pointing which is not possible.

\item[\eqref{C2}]
By Proposition \ref{prop:dim_of_Q(p)} we know that the Hausdorff dimension of the conjugate cut locus $Q(p)$ does not exceed $n-2$. If we can prove the same for the set $L(p)$ of a-typical cut points the claim \eqref{C2} follows from Lemma \ref{lem:Haus_dim_prop}.

Recall that $L(p)\subset (\locus(p)\setminus Q(p))$ is the set of points in $M$ that can be connected to $p$ with more than two distance minimizing geodesics of $M$.  Let $q \in L(p)$ and define $k_p(q)\in \N$ to be the number of distance minimizing geodesics from $p$ to $q$.  Then we choose vectors $w_1,\ldots,w_{k_p(q)} \in M_p$ and their respective neighborhoods $U_i \in T_pN$ such that for each $i \in \{1,\ldots,k_p(q)\}$  
\[
\exp_p(w_i)=q, \quad \hbox{ and } \exp_{p}\colon U_i \to \tilde V
\]
is a diffeomorphism on some open set $\tilde V \subset N$.  Let $Y_i$ for $i \in  \{1,\ldots,k_p(q)\}$ be the distance vector fields related to the $U_i$ and $p$. Then we define a collection of smooth functions 
\[
\rho_{ij}\colon \tilde V \to \R, \quad \rho_{ij}(x)=\|Y_i(x)\|_{g}-\|Y_j(x)\|_{g}, \quad i,j \in  \{1,\ldots,k_p(q)\}.
\]
By the proof of Lemma \ref{prop:hyper_surface} it holds that the sets 
$
K_{ij}:=\rho_{ij}^{-1}\{0\},  \: \hbox{ for } i<j 
$
are smooth hyper-surfaces of $N$ that contain $q$.  Also by \cite[Proposition 2.6]{ozols1974cut} it holds that the sets 
\[
K_{i,j,k}:=K_{ik}\cap K_{jk}, \quad \hbox{ for } i<j<k
\]
are smooth submanifolds of $N$ of co-dimension two.  Next we set $K(q):=\bigcup_{i<j<k}K_{i,j,k}$ and claim that
\begin{equation}
\label{eq:L(p)=K(q)}
L(p)\cap \tilde V = K(q)\cap M. 
\end{equation}
Since the sets $K_{i,j,k}$ are smooth sub-manifolds of dimension $n-2$ their Hausdorff dimension is also $n-2$. Thus the equation \eqref{eq:L(p)=K(q)} and Lemma \ref{lem:Haus_dim_prop} imply that Hausdorff dimension of $L(p)$ does not exceed $n-2$. 

\medskip
Finally we verify the equation \eqref{eq:L(p)=K(q)}. If $x \in L(p)\cap \tilde V$ it holds there are at least three distance minimizing geodesics of $M$ connecting $p$ to $x$.  Thus there are $1\leq i<j<k\leq k_{p}(q)$ so that 
\[
\|Y_i(x)\|_g=\|Y_j(x)\|_g=\|Y_k(x)\|_g=d_M(p,x),
\]
which yields
\[
\rho_{ik}(x)=\rho_{jk}(x)=0,\quad \hbox{and $x \in K_{i,j,k} \subset K(q) $}.
\]

If $x \in K(q)\cap M$ then $x \in K_{i,j,k}\cap M$ for some $i<j<k$.  Thus by the proof of Lemma \ref{prop:hyper_surface} it holds that there are at least three distance minimizing geodesics of $M$ connecting $p$ to $x$. Therefore $x \in L(p)\cap \tilde V$.  

\item[\eqref{C3}]
Since we can write the cut locus of $p$ as a disjoint union 
$
\locus(p)=T(p)\cup L(p) \cup Q(p)
$
the parts \eqref{C1} and \eqref{C2} in conjunction with Lemma \ref{lem:Haus_dim_prop} yield the claim of part \eqref{C3}.

\item[\eqref{C4}]
Since $\p M$ is a smooth hyper-surface of a $n$-dimensional Riemannian manifold $M$ we have by part \eqref{C1} that $T(p)\cap \p M$ is a smooth sub-manifold of dimension $n-2$, thus it has the Hausdorff-dimension $n-2$. Also by part \eqref{C3} we know that the Hausdorff dimension of $L(p)\cap Q(p)$ does not exceed $n-2$. We have proven that the Hausdorff dimension of the closed set
$
\locus(p)\cap \p M
$
does not exceed $n-2$.  Since the boundary of $M$ has the Hausdorff-dimension $n-1$ it follows that $\p M \setminus \locus(p)$ is open and dense in $\p M$. The density claim follows from the observation that by Lemma \ref{lem:Haus_dim_prop} the set  $\locus(p)\cap \p M$ cannot contain any open subsets of $\p M$ as their Hausdorff dimension is $n-1$. 
\end{itemize}
\end{proof}

We are ready to prove Theorem \ref{prop:dist_smoothness}.
\begin{proof}[Proof of Theorem \ref{prop:dist_smoothness}]
The proof follows from Proposition \ref{cor:smoothness_of_d} and Theorem \ref{thm:cut_locus}.
\end{proof}

\section{Reconstruction of the manifold}
\label{sec:reconst}

\subsection{Geometry of the measurement region} 
\label{sec:metric_on_Gamma}
In this section we consider only one Riemannian manifold $(M,g)$ that satisfies the assumptions of Theorem \ref{prop:dist_smoothness} and whose partial travel time data \eqref{eqn:data} is known. 
Let $\nu(z)$ be the outward pointing unit normal vector field at $z\in \dM$. The inward pointing bundle at the boundary is the set
\[
\din TM = \{ (z,v) \in TM ~|~  z \in \dM, \langle v, \nu(z)\rangle_g <0\}. 
\]
We restrict our attention to the vectors that are inward pointing and of unit length: $\din SM = \{ (z,v) \in \din TM ~:~  \|v\|_g = 1\}$. We emphasize that this set or its restriction on the open measurement region $\Gamma \subset \p M$ is not \textit{a priori} given by the data \eqref{eqn:data}. Our first task is to recover a diffeomorphic copy of this set.
We consider the orthogonal projection 
\begin{equation}
h: \din SM \to T\dM, \qquad h(z,v) = v - \langle v,\nu(z)\rangle_g \nu(z),
\label{eqn:h}
\end{equation}
and denote the set, that contains the  image  of $h$, as  $P(\dM) =\{ (z,w) \in T \dM ~:~ \|w\|_g <1\} $. It is straightforward to show that the map $h$ is a diffeomorphism onto $P(\dM)$. 



For the rest of this section we will be considering the vectors in $P(\dM)$, and with a slight abuse of notation, each vector $(z,v) \in P(\dM)$ represents an inward-pointing unit vector at $z$. In the next lemma we show that the data \eqref{eqn:data} determines the restriction of $P(\dM)$ on $\Gamma$.

\begin{lemma}
\label{lem:ball_at_the_boundary}
Let Riemannian manifold $(M,g)$ be as in Theorem \ref{prop:dist_smoothness}. The first fundamental form $g|_{\Gamma}$ of $\Gamma$ and the set 
\[
P(\Gamma)=\{(z,v) \in T\p M:\: z\in \Gamma, \: \|v\|_g<1\}
\]
can be recovered from data \eqref{eqn:data}.
\end{lemma}
\begin{proof}
Let $(z,v) \in T\Gamma$. We choose a smooth curve $c\colon (-1,1) \to \Gamma$ for which
$c(0)=z,$ and  $\dot{c}(0)=v$. 
Since the boundary $\p M$ of $M$ is strictly convex the inverse function of the exponential map $\exp_z$ is smooth and well defined near $z$ on $M$. In addition, we have that
\[
r_z(c(t))=d(z,c(t))=\|\exp^{-1}_z(c(t))\|_g.
\]
We set
$
\tilde c(t)=\exp^{-1}_z(c(t))\in T_{z} M.
$
As the differential of the exponential map at the origin is an identity operator we get
$
\tilde c(0)=0, \text{ and } \dot{\tilde c}(0)=v\in T_zM.
$
From here the continuity of the norm yields
\begin{equation}
\label{eq:g_on_boundary}
\lim_{t \to 0}\frac{r_{z}(c(t))}{|t|}
=\lim_{t \to 0}\left\|\frac{\tilde c(0)-\tilde c(t)}{t}\right\|_g
=\left\|\dot{\tilde c}(0)\right\|_g
=\|v\|_g.
\end{equation}
By the data \eqref{eqn:data} and the choice of the path $c(t)\in \Gamma$ we know the left hand side of equation \eqref{eq:g_on_boundary}. Therefore we have recovered the length of an arbitrary vector $(z,v)\in T\Gamma$. Moreover, the set $P(\Gamma)$ is recovered.

Since we know the unit sphere 
$
\{v \in T_{z}\p M: \: \|v\|_g=1\}
$
for each $z \in \Gamma$ the reconstruction of the first fundamental form of $\Gamma$ can be carried out as explained in the next lemma.
\end{proof}

\begin{lemma}
\label{lem:inner_prod_from_sphere}
Let $(X,g)$ be a finite dimensional inner product space. Let $a>0$ and 
$
S(a):= \{v\in X: \: \|v\|_g=a\}.
$
Then any open subset  $U$ of $S(a)$ determines the inner product $g$ on $X$.
\end{lemma}
\begin{proof}
This proof is the same as the one in \cite[Lemma 3.33]{Katchalov2001} and thus omitted here.  
\end{proof}
 
Let $p_0 \in M$. By Theorem \ref{prop:dist_smoothness} we can find a boundary point $z_0 \in \Gamma$ and neighborhoods $\upo$ and $\vzo$ for $p_0$ and $z_0$ respectively such that the distance function $d(\cdot,\cdot)$ is smooth in the product set $\upo \times \vzo$. 
For each $z\in \Gamma\cap \vzo$ we let $\gamma_z$ be the unique distance minimizing unit speed geodesic from $p_0$ to $z$. If we decompose the velocity of the geodesic $\gamma_z$ at $r_{p_0}(z)$ into its tangential and normal components to the boundary, then the tangential component coincides with the boundary gradient of the travel time function $r_{p_0}$ at $z$. For this vector field we use the notation
$
\grad_{\dM}r_{p_0}(z)\in  P_z(\Gamma).
$
Furthermore, by Lemma \ref{lem:ball_at_the_boundary} we have recovered the metric tensor of the measurement domain $\Gamma\subset \p M$. Thus we can compute $\grad_{\dM}r_{p_0}(z)$ whenever the respective travel time function $r_{p_0}$ is differentiable on $\Gamma$.

 \subsection{Topological reconstruction}
\label{sec:topology}

We first show that the data \eqref{eqn:data} separates the points in the manifold $M$.

\begin{lemma} 
Let $(M,g)$ be as in Theorem \ref{prop:dist_smoothness}. Let $\Gamma \subset \p M$ be open and $p_1,p_2 \in M$ be such that $r_{p_1}(z) = r_{p_2}(z)$ for all $z \in \Gamma$, then $p_1 = p_2$. 
\label{lem:unique}
\end{lemma}

\begin{proof}
First we choose open and dense subsets $W_1,W_2 \subset \dM$ for the points $p_1$ and $p_2$ as we have for the point $p_0$ in Theorem \ref{prop:dist_smoothness}. Then we choose any point $z_0 \in W_{p_1}\cap W_{p_2} \cap \Gamma$, neighborhoods $U_{p_1}$ of $p_1$, $U_{p_2}$ of $p_2$ and $V_{p_1}$, $V_{p_2}$ of $z_0$ as we have for $p_0$ in Theorem \ref{prop:dist_smoothness}. Thus the distance function $d(\cdot,\cdot)$ is smooth in the product sets $U_{p_1}\times V$ and $U_{p_2}\times V$, where $V=V_{p_1}\cap V_{p_2}$ is an open neighborhood of $z_0$. Moreover for each $(p,z)\in U_{p_i}\times V, \: i \in \{1,2\}$ there exists a unique distance minimizing geodesic of $M$ connecting $p$ to $z$. 

If $\gamma_i$ is the distance minimizing geodesic from $p_i$ to $z_0$ for $i = \{1,2\}$ then by the discussion preceding this lemma we have that $\grad_{\dM}r_{p_i}(z_0)$ represents the tangential component of $\dot \gamma_i$ at $r_{p_i}(z_0)$. Since $r_{p_1}=r_{p_2}$  the tangential components of $\dot \gamma_1$ and $\dot \gamma_2$ are the same. Since the velocity vectors of $\dot \gamma_i$ at $r_{p_i}(z_0)$ have unit length, they must also coincide.
We get
\[
z_0=\gamma_1(r_{p_1}(z_0))=\gamma_2(r_{p_2}(z_0)) \quad \text{ and } \quad \dot \gamma_1(r_{p_1}(z_0))=\dot \gamma_2(r_{p_2}(z_0)).
\]
Thus the geodesics $\gamma_1$ and $\gamma_2$ agree and we have $p_1=p_2$.
\end{proof}

We are now ready to reconstruct the topological structure of $(M,g)$ from the partial travel time data \eqref{eqn:data}. Let $B(\Gamma)$ be the collection of all bounded functions $f\colon \Gamma \to \R$ and $\|\cdot\|_\infty$ the supremum norm of $B(\Gamma)$. Thus $(B(\Gamma),\|\cdot\|_\infty)$ is a Banach space. Since $(M,g)$ is a compact Riemannian manifold each travel time function $r_p$, for $p \in M$, is bounded by the diameter of $M$, which is finite. Thus
\[
\{r_p=d(p,\cdot)\colon \Gamma \to [0,\infty)|~ p \in M\}\subset B(\Gamma),
\]
and the map
\begin{equation}
\label{eq:map_R}
R:(M,g) \to  (B(\Gamma), \| \cdot \|_{\infty}), \quad R(p)=r_p
\end{equation}
is well defined.

\begin{proposition}
\label{lemma:homeo}
Let Riemannian manifold $(M,g)$ be as in Theorem \ref{prop:dist_smoothness}. The map $R$ as in \eqref{eq:map_R} is a topological embedding. 
\end{proposition}
\begin{proof}
By Lemma \ref{lem:unique}, we know that the map $R$ is injective, and by the triangle inequality we get that it is also continuous.
%
%
Let $K$ be a closed set in $M$. Since $M$ is a compact Hausdorff space the set $K$ is compact. Since the image of a compact set under a continuous mapping is compact, it follows that $R(K)$ is closed. This makes $R$ a closed map and thus a topological embedding. 
\end{proof}


\subsection{Boundary Determination}
We recall that the data \eqref{eqn:data} only gives us the subset $\Gamma$ of the boundary, and we do not know yet if the travel time function $r\in R(M)$ is related to an interior or a boundary point of $M$. In this subsection we will use the data  \eqref{eqn:data} to determine the boundary of the unknown manifold $M$ as a point set. However, due to Proposition \ref{lemma:homeo} we may assume without loss of generality that the topology of $M$ is known. Also the set $P(\Gamma)$, as in Lemma \ref{lem:ball_at_the_boundary}, is known to us. 

Let $(z,v) \in P(\Gamma)$, and define the set,
\begin{equation}
\label{eqn:sigma}
\begin{split}
\sigma(z,v) =& \{ p \in M ~|~\text{the point $p$ has a neighborhood $U\subset M$ such that,} 
\\
& \: r_q \colon \Gamma \to \R \text{ is smooth near $z$ for every $q \in U$},
\\
& q \mapsto \grad_{\dM}r_q(z) 
\text{ is continuous in $U$},
\\
& \grad_{\dM} r_p(z) = -v\} \cup \{ z \},
\end{split}
\end{equation}
where $\grad_{\dM}r_p(z)$ is the boundary gradient of $r_p$ at $z \in \Gamma$. 
We recall that by Proposition \ref{lemma:homeo} we know the topology of $M$, and by Lemma \ref{lem:ball_at_the_boundary} we know the geometry of $\Gamma$. These in conjunction with the data \eqref{eqn:data} imply that we can recover the set $\sigma(z,v)$ for every $(z,v) \in P(\Gamma)$. 
In the next lemma we generalize the result \cite[Lemma 2.9]{lassas2015determination} and relate $\sigma(z,v)$ to the maximal distance minimizing segment of the geodesic $\gamma_{z,v}$. 

\begin{lemma}
Let $(z,v) \in P(\Gamma)$ then $\overline{\sigma(z,v)} = \gamma_{z,v}([0,\cut(z,v)])$.
\label{lemma:sigmaClosure}
\end{lemma}

\begin{proof}
Let $t \in [0, \cut(z,v))$ and define $y: = \gamma_{z,v}(t)$.
Thus $y$ is not in the cut locus of $z$ (see equation \eqref{eq:cut_locus_real_def}). By Proposition \ref{cor:smoothness_of_d} there exist  neighborhoods $U,V\subset M$ of $y$ and $z$ respectively, having the property that the distance function $d(\cdot, \cdot)$ is smooth in the product set $U \times V$. Therefore the function $r_q(\cdot)=d(q,\cdot)|_{\Gamma}$ is smooth near $z$ for any $q \in U$. Furthermore, $\grad_{\dM}r_p(z) = -v$, and the function $p \mapsto \grad_{\dM} r_p(z)$ is continuous in $U$. 
Therefore $y$ is in $\sigma(z,v)$ and the inclusion $\gamma_{z,v}([0,\cut(z,v))) \subseteq \sigma(z,v)$ is true. This gives $\gamma_{z,v}([0,\cut(z,v)]) \subseteq \overline{\sigma(z,v)}$.

Let $p \in \sigma(z,v)$, then $r_p(z)$ is smooth in a neighborhood of $z$ and $\grad_{\dM}r_p(z) = -v$. Thus $\gamma_{z, v}$ is the unique distance minimizing geodesic connecting $z$ to $p$. Since the geodesic $\gamma_{z,v}$ is not distance minimizing beyond the interval $[0,\cut(z,v)]$ we have $p \in \gamma_{z,v}([0,\cut(z,v)])$ and therefore
$
\overline{\sigma(z,v)}\subset \gamma_{z,v}([0,\cut(z,v)]).
$
\end{proof}

We set
\begin{equation}
    T_{z,v}:= \sup_{p \in \sigma(z,v)} r_p(z) = \sup_{p \in \sigma(z,v)} d(p,z), \qquad \text{ for }(z,v) \in P(\Gamma).
    \label{eqn:Tzv}
\end{equation}
Notice that this number is determined entirely by the data \eqref{eqn:data}, as opposed to $\cut(z,v)$ which requires our knowledge of when the geodesics were distance minimizing. By the following corollary, whose proof is evident, these two numbers are the same.

\begin{corollary}
For any $(z,v) \in P(\Gamma)$ we have that $ T_{z,v}= \cut(z,v) $.
\label{cor:cut=Tzv}
\end{corollary}


We will use the sets $\sigma(z,v)$, for $(z,v) \in P(\Gamma)$ to determine the boundary $\dM$ of $M$. Since the topology of $M$ is known by Proposition \ref{lemma:homeo}, we can determine the topology of these $\sigma$ sets from the data. The next lemma shows if $\sigma(z,v)$ is closed then $\gamma_{z,v}(T_{z,v})$ is on the boundary of $M$.

\begin{lemma}
Let $(z,v) \in P(\Gamma)$. If $\sigma(z,v)$ is closed then $T_{z,v} = \exit(z,v)$.
\label{thm: Tzv=exit}
\end{lemma}
\begin{proof}
By the definition of $T_{z,v}$ we must have $T_{z,v} \leq \exit(z,v)$. 

Suppose that $T_{z,v} < \exit(z,v)$. From Corollary \ref{cor:cut=Tzv} then we also know 
\begin{equation}
    \label{eq:T_less_than_exit}
\cut(z,v) = T_{z,v} < \exit(z,v)
\end{equation}
Let $p = \gamma_{z,v}(T_{z,v})$, and by Lemma \ref{lemma:sigmaClosure} it holds that $p \in  \overline{\sigma(z,v)} = \sigma(z,v)$. 
Since $p \in \locus(z)$ we have by Lemma \ref{lem:cut_time}, that there either exists a second distance minimizing geodesic from $z$ to $p$ or $p$ is a conjugate point to $z$ along $\gamma_{z,v}$. 
In the first case let $w \in P(\Gamma)$ such that $\gamma_{z,w}$ is another unit-speed distance minimizing geodesic from $z$ to $p$. 
We note that $T_{z,v} = \cut(z,w)$.

Let $U$ be a neighborhood of $p$ as in \eqref{eqn:sigma}. 
We consider a sequence $t_i\in [0,T_{z,v}], \: i \in \N$ such that $t_i \to T_{z,v}$ as $i \to \infty$. Then for sufficiently large $i$ the points $p_i = \gamma_{z,v}(t_i)$ and $q_i = \gamma_{z,w}(t_i)$ are in $U$ and converge to $p$.
By the continuity of the boundary gradient in $U$ we have $\grad_{\dM} r_{p_i}(z) \to \grad_{\dM} r_p(z)$ and $\grad_{\dM} r_{q_i}(z) \to \grad_{\dM} r_p(z)$, when $i \to \infty$.  However, by construction $\grad_{\dM} r_{p_i}(z)=-v$ while $\grad_{\dM} r_{q_i}(z)=-w$ for all $i \in \N$. 
Thus $\grad_{\dM}r_p(z)$ has multiple values, and  $r_p$ is not differentiable at $z$, contradicting that $p \in \sigma(z,v)$.

If the second case is  valid, and since $p\in M^{int}$, we get by a similar proof as in \cite[Theorem 2.1.12]{klingenberg} that the exponential map $\exp_{z}$ is not a local injection at $T_{z,v}v \in T_zM$. From here  \cite[Theorem 2.1.14]{klingenberg} implies that there is a sequence of points $(p_i)_{i=1}^\infty$ in $M^{int}$, that converges to $p$ and can be connected to $z$ by at least two distance minimizing geodesics. 
By the same argument as in the previous case, $r_{p_i}$ is not differentiable at $z$ for any $i \in \N$, which contradicts the fact that $p \in \sigma(z,v)$. 
Thus inequality \eqref{eq:T_less_than_exit} cannot occur and we must have $T_{z,v} = \exit(z,v)$.
\end{proof}

\begin{lemma}
Let $p_0 \in \dM$ and $z_0 \in \Gamma$, $\upo$, and $\vzo$ be as in Theorem \ref{prop:dist_smoothness}. For every $p \in \upo$ we denote $\eta(p) = -\grad_{\dM}r_p(z_0)$. There exists a neighborhood $\upo'\subseteq \upo$ of $p_0$ such that for all $p \in \upo'$ we have that $p$ is in the closed set 
$ \sigma(z_0,\eta(p))$. 
   \label{thm:bndry_tzv=texit}
\end{lemma}
\begin{proof}
By these assumptions, $d(\cdot, \cdot)$ in $\upo\times \vzo$ is smooth.
Define $v_0 = \eta(p_0)$ and  $t_0 = \exit(z_0, v_0)$, then $p_0 = \exp_{z_0}(t_0v_0)$. Since $z_0$ was chosen to be a point outside the cut locus of $p_0$, these points are not conjugate to each other along the geodesic $\gamma_{z_0,v_0}$ connecting them. Therefore the differential $\mathrm{D}\exp_{z_0}$ of the exponential map is invertible at $t_0v_0\in T_{z_0}M$. From here the claim follows from the Inverse function theorem for $\exp_{z_0}$ near $t_0v_0$, the continuity of the exit time function 
on the non-trapping part of $SM$,
and the inequality 
\[
r_p(z_0)=\|\exp_{z_0}^{-1}(p)\|_{q}\leq \exit(z_0,\eta(p)), \quad \text{for } p \in \upo.
\]
We omit the further details. 

\end{proof}

\begin{corollary}
Let $p_0 \in \dM$, $z_0 \in \Gamma$ and $\upo'$ be as in Lemma \ref{thm:bndry_tzv=texit}. If we denote $\eta(p) = -\grad_{\dM}r_p(z_0)$ then $T_{z_0,\eta(p)}$ is smooth for all $p \in \upo'$.
\label{cor:Tzv_smooth}
\end{corollary}

\begin{proof}
Since the exit time function is smooth on those $(z,v)\in \din SM$ that satisfy $\exit(z,v)<\infty$ we only need to show that
$
T_{z,\eta(p)}=\exit(z_0,\eta(p)).
$
This equation follows from lemmas \ref{thm: Tzv=exit} and  \ref{thm:bndry_tzv=texit}.
\end{proof}

We are now ready to determine the boundary of $M$ from the data \eqref{eqn:data}. 
\begin{proposition}
Let $(M,g)$ be as in Theorem \ref{prop:dist_smoothness} and $p_0 \in M$. Then $p_0 \in \dM$ if and only if there exists $(z,v) \in P(\Gamma)$ such that  $p_0 \in \sigma(z,v)$ and $r_{p_0}(z) = T_{z,v}$. 
\label{thm:bndry_defining}
\end{proposition}

\begin{proof}
If $p_0 \in \dM$ then we get from Lemma \ref{thm:bndry_tzv=texit} that there exists $(z_0, v) \in P(\Gamma)$ such that $p_0$ is in the closed set $\sigma(z_0, v)$. By Lemma \ref{thm: Tzv=exit} we have  $T_{z_0, v} = \exit(z_0, v)$. Firstly the strict convexity of $\dM$ implies that each geodesic has at most two boundary points. Secondly since $p_0 \neq z_0$ are both boundary points contained in $\sigma(z_0, v)$, which is a trace of a distance minimising geodesic, it follows that 
$
 T_{z_0,v}
 = r_{p_0}(z_0).
$

To show the reverse direction, let $(z,v) \in P(\Gamma)$ be such that  $p_0 \in \sigma(z,v)$ and $T_{z,v} = r_{p_0}(z)$. Thus 
$
\gamma_{z,v}([0,r_{p_0}(z)]) 
\subseteq \sigma(z,v),
$
and it follows from Lemma \ref{lemma:sigmaClosure} and Corollary \ref{cor:cut=Tzv} that 
$\sigma(z,v)$ is closed. 
By Lemma \ref{thm: Tzv=exit}, the closedness of $\sigma(z,v)$ implies $T_{z,v} = \exit(z,v)$. Thus, $r_{p_0}(z) = \exit(z,v)$, making $p_0 \in \dM$. 
\end{proof}

\subsection{Local Coordinates}
By Proposition \ref{thm:bndry_defining} we have reconstructed the boundary $\p M$ of the smooth manifold $M$. In this section we use the partial travel time data \eqref{eqn:data} to construct two local coordinate systems for $p_0\in M$. Since $M$ has a boundary, we need different coordinates systems based on whether $p_0 \in M^{int}$ or $p_0 \in \dM$. 
 
\begin{proposition}
Let $(M,g)$ be as in Theorem \ref{prop:dist_smoothness}. Let $p_0 \in M^{int}$, and choose $z_0 \in \Gamma$, $\upo$, and $\vzo$ as in Theorem \ref{prop:dist_smoothness}. Let the map $\alpha: \upo \to P_{z_0}(\Gamma) \times \R$ be defined as 
\begin{equation}
    \label{eq:func_phi_z}
\alpha(p) = (-\grad_{\dM}r_p(z_0) ,  r_p(z_0) ).
\end{equation}
This map is a diffeomorphism onto its image $\alpha(\upo)\subset  P_{z_0}(\Gamma) \times \R$.
\label{thm:int_coords}
\end{proposition}
\begin{proof}
Since the distance function $d(\cdot,\cdot)$ is smooth in $\upo \times V_{p_0}$ also the function $\alpha$ is smooth on $\upo$. By a direct computation we see that the inverse function of $\alpha$, is given as,
\[
\alpha^{-1}(v,t)  = \exp_{z_0}\left(th^{-1}(v) \right), \qquad \text{ for } (v,t) \in P_{z_0}(\Gamma) \times \R.
\]
where $h\colon \din S_{z_0}M\to P_{z_0}(\Gamma)$, is the orthogonal projection given in \eqref{eqn:h}. By the smoothness of $h^{-1}$ and the exponential map, it follows that $\alpha^{-1}$ is smooth. Thus, $\alpha$ is a diffeomorphism onto its image, which is open in $P_{z_0}(\Gamma) \times \R$. 
\end{proof}

In particular, the function $\alpha$, in \eqref{eq:func_phi_z}, gives a local coordinate system near the interior point $p_0$. In order to define a coordinate system for a point at the boundary we will adjust the last coordinate function of $\alpha$ to be a boundary defining function. 

\begin{proposition}
Let $(M,g)$ be as in Theorem \ref{prop:dist_smoothness}. Let $p_0 \in \dM$ and choose $z_0 \in \Gamma$, and $\upo'$ as in Lemma \ref{thm:bndry_tzv=texit}. Let $\eta(p):=-\grad_{\p M}r_p(z_0)$ and  $\beta_{z_0}: \upo' \to P_{z_0}(\Gamma) \times [0,\infty)$ 
be defined as 
\begin{equation}
    \label{eq:func_psi_z}
\beta(p) = (\eta(p), T_{z_0, \eta(p)} - r_p(z_0) ).
\end{equation}
This map is a diffeomorphism onto its image $\beta(\upo')\subset P_{z_0}(\Gamma) \times [0,\infty)$.
 \label{thm:bndry_coords}
\end{proposition}

\begin{proof}
Since the distance function $d(\cdot,\cdot)$ is smooth in $\upo' \times V_{p_0}$ and $p_0\in \dM$ we have by Corollary \ref{cor:Tzv_smooth}, that the map  $T_{z_0,\eta(p)}$ is smooth for all $p \in \upo'$. Thus $\beta$ is smooth in $\upo'$. Again by a direct computation we get that the inverse function of $\beta$ is given as
\[
 \beta^{-1}(v,t) = \exp_{z_0}\left(\left(T_{z_0, v} - t\right)h^{-1}(v)\right) \qquad \text{ for } (v,t) \in P_{z_0}(\Gamma) \times [0,\infty).
\]
By the local invertibility of the exponential map $\exp_{z_0}$ at $r_{p_0}(z_0)h^{-1}(\eta(p_0))\in T_{z_0}M$ and the equation $r_p(z_0)=\|\exp_{z_0}^{-1}(p)\|_{g}$ for $p \in \upo'$, the set $\eta(\upo')\subset P_{z_0}(\Gamma)$ is open and the function $v\mapsto T_{z_0,v}$, in this set is smooth, making  $\beta^{-1}$ smooth. Thus, $\beta$ is a diffeomorphism onto its image, which is open in $P_{z_0}(\Gamma) \times [0,\infty)$. 

Finally by Proposition \ref{thm: Tzv=exit} we get that $T_{z_0, \eta(p)}-r_p(z_0)=0$ if and only if $p \in \upo'\cap \p M$. Thus this function defines the boundary. 
\end{proof}

Combining the results of Propositions \ref{thm:int_coords} and \ref{thm:bndry_coords}, we know that for $p_0\in M$, either the function $\alpha$ as in \eqref{eq:func_phi_z} or the function $\beta$ as in \eqref{eq:func_psi_z}, gives a smooth local coordinate system. Moreover these maps can be recovered fully from the data \eqref{eqn:data}. As these two types of coordinate charts cover $M$ the smooth structure on $M$ is then the same as the maximal smooth atlas determined by these coordinate charts \cite[Proposition 1.17]{lee2018introduction}.

\subsection{Reconstruction of the Riemannian Metric}
So far we recovered both the topological and smooth structures of the Riemannian manifold $(M,g)$ from the data \eqref{eqn:data}. In this section we recover the Riemannian metric $g$. We recall that by Lemma \ref{lem:ball_at_the_boundary} we know the first fundamental form of $\Gamma$. 

In order to recover the metric on $M$ we consider the distance function
\[
d(p,z)=r_p(z),  \quad \text{for }(p,z) \in M \times \Gamma,
\]
which we have recovered by Proposition \ref{lemma:homeo}. 
Let $p_0 \in M$. By Theorem \ref{prop:dist_smoothness} we can choose $z_0\in \Gamma$ and neighborhoods $\upo$ and $\vzo$ for $p_0$ and $z_0$ respectively such that the distance function $d(p,z)$ for $(p,z)\in \upo\times \vzo$ is smooth. Thus the map 
\begin{equation}
\label{eq:grad_map}
H_{p_0}\colon \vzo\cap \Gamma \to T^\ast_{p_0}M,\quad H_{p_0 }(z)=\mathrm{D} d(p_0,z)
\end{equation}
is well defined and smooth. Here $\mathrm{D}$ stands for the differential of the distance function $d(p,z)$ with respect to the $p$ variable in the open set $\upo\subset M$ and $T^\ast_{p_0}M$ is the cotangent space at $p_0$. As we have recovered the smooth structure of $M$ we can find $H_{p_0}$. 

For $z \in \vzo$ the gradient $\text{grad}_pd(p,z)$ for $p \in \upo$ is the velocity of the distance minimizing unit speed geodesic from $z$ to $p$ (see for instance \cite[theorems 6.31, 6.32]{lee2018introduction}). In particular the map 
\[
\tilde{H}_{p_0} \colon (\vzo \cap \Gamma) \ni z \to \text{grad}_pd(p_0,z)\in S_{p_0}M
\]
is well defined and satisfies
$
\tilde{H}_{p_0}(z)=H_{p_0}(z)^\sharp,
$
where $\sharp \colon T^\ast_{p_0} M \to T_{p_0}M$ is the musical isomorphism, raising the indices, given in any local coordinates near $p_0$ as
$
(\xi^\sharp)^i=g^{ij}(p_0)\xi_j.
$
Note that the inverse of $\sharp$ is given by $\flat:T_{p_0}M \to T_{p_0}^*M$, that lowers the indices. 
Although we know the map $H_{p_0}$, we do not know its sister map $\tilde{H}_{p_0}$. 

\begin{lemma}
\label{lem:open_subset_unit_sphere}
Let $p_0\in M$. Let $z_0\in \Gamma$, $\upo$ and $\vzo$ be as in Theorem \ref{prop:dist_smoothness}. Then the image of the map $H_{p_0}$, as in \eqref{eq:grad_map}, 
is contained the unit co-sphere
\[
S^\ast_{p_0}M:=\{\xi \in T_{p_0}^\ast M: \: \|\xi\|_{g^{-1}}=1\},
\]
and has a nonempty interior.
\end{lemma}
\begin{proof}
Let $v\in T_{p_0}M$ such that $\exp_{p_0}(v)=z_0$. Since $\flat\colon T_{p_0}M\to T^\ast_{p_0}M$ is a linear isomorphism that preserves the inner product, the claim holds due to the local invertibility of the exponential map $\exp_{p}$ near $v$, the equality $d(p_{0},z)=\|\exp_{p_0}^{-1}(z)\|_g$, which is true for all $z \in V_{p_0}$, and the continuity of the exit time function near $\frac{v}{\|v\|_g}\in S_{p_0}M$. We omit the further details. 
\end{proof}
 
Finally Lemma \ref{lem:inner_prod_from_sphere} in conjunction with the previous lemma lets us recover the inverse metric $g^{-1}(p_0)$ and thus the metric $g_{p_0}$. This is formalized in the proposition below.

\begin{proposition}
Let $(M,g)$ be as in Theorem \ref{prop:dist_smoothness} and $p_0 \in M$. The data \eqref{eqn:data} determines the metric tensor $g$ near $p_0$ in the local coordinates given in Propositions  \ref{thm:int_coords} and \ref{thm:bndry_coords}.
 \label{prop:metric}
\end{proposition}

\begin{proof}
Let $z_0\in \Gamma$, $\upo$ and $\vzo$ be as in Theorem \ref{prop:dist_smoothness}. By Proposition \ref{thm:bndry_defining} we can tell whether $p_0$ is an interior or a boundary point. Based on this we choose local coordinates of $p_0$ as in Proposition \ref{thm:int_coords} or as in Proposition \ref{thm:bndry_coords}. Then we consider the function $H_{p_0}$ given in the equation \eqref{eq:grad_map}. By Lemma \ref{lem:open_subset_unit_sphere} we know that image of the function $H_{p_0}$ contains an open subset of $S_{p_0}^*M$. 

From here, by applying Lemma \ref{lem:inner_prod_from_sphere} we determine the inverse metric $g^{ij}(p_0)$ in the aforementioned coordinates.
Finally taking the inverse of $g^{ij}(p_0)$ determines $g_{ij}(p_0)$. As this procedure can be done for any point $p\in M$, which is close enough to $p_0$, we have recovered the metric $g$ near $p_0$ in the appropriate local coordinates. 
\end{proof}

\section{The proof of Theorem \ref{thm:main}}
\label{sec:proof}

Let Riemannian manifolds $(M_1,g_1)$ and $(M_2,g_2)$ be as in Theorem \ref{thm:main}. We recall that the partial travel time data of these manifolds coincide in the sense of Definition \ref{def:data}. Let $(B(\Gamma_i), \|\cdot\|_\infty)$, for $i \in \{1,2\}$, be the Banach space of bounded real valued functions on $\Gamma_i$.  We set a mapping
\begin{equation}
F:B(\Gamma_1) \to B(\Gamma_2), \qquad F(f) = f \circ \phi^{-1},
\label{eqn:pre-compose}
\end{equation}
where $\phi$ is the diffeomorphism from $\Gamma_1$ to $\Gamma_2$. By the triangle inequality we have that $F$ is a metric isometry whose inverse mapping is given by $ F^{-1}(h) = h \circ \phi.$ Taking $R_i:(M_i, g_i) \to (B(\Gamma_i), \|\cdot\|_\infty)$, as in the equation \eqref{eq:map_R}, we have by the equation \eqref{eq:equiv_of_data} in Definition \ref{def:data} that 
\[
F(R_1(M_1))=R_2(M_2).
\]
Therefore we get from Proposition \ref{lemma:homeo} that the map
\begin{equation}
\Psi: (M_1,g_1) \xrightarrow{R_1} (B(\Gamma_1), \|\cdot \|_\infty ) \xrightarrow{F} (B(\Gamma_2), \|\cdot \|_\infty) \xrightarrow{R_2^{-1}}( M_2, g_2),
\label{eqn:diffeo_between_manifolds}
\end{equation}
is a well defined homeomorphism, that satisfies the equation
\begin{equation}
    \label{eq:distances_agree}
d_2(\Psi(x),\phi(z))
=F(d_1(x,\cdot))(\phi(z))
=d_1(x,z),\quad \text{ for all } (x,z)\in M_1\times \Gamma_1.
\end{equation}
Here $d_i(\cdot,\cdot)$ is the distance function of $(M_i,g_i)$. The goal of this section is to show that $\Psi$ is a Riemannian isometry. 
In the following lemma we show first that $\phi$ preserves the Riemannian structure of the measurement regions. 

\begin{lemma}
Let  Riemannian manifolds $(M_1,g_1)$ and $(M_2,g_2)$ be as in Theorem \ref{thm:main}.
Then $\Psi|_{\Gamma_1}=\phi$ and $\phi\colon (\Gamma_1,g_1)\to (\Gamma_2,g_2)$ is a Riemannian isometry. 
\label{lemma:iso_on_gammas}
\end{lemma}

\begin{proof}
Let $z_1$ be in $\Gamma_1$. From equation \eqref{eq:distances_agree} we get $
d_2(\Psi(z_1),\phi(z_1))=0.
$
Thus $\Psi(z_1)=\phi(z_1)$ and we have verified the first claim  $\Psi|_{\Gamma_1}=\phi$. It follows from the proof of Lemma \ref{lem:ball_at_the_boundary} and equation \eqref{eq:distances_agree} that
$
\|\mathrm{D}\phi v\|_{g_2}=\|v\|_{g_1}, \text{ for all } v \in T\Gamma_1.
$
Then the polarization identity implies that the differential $\mathrm{D}\phi$ of $\phi$ also preserves the first fundamental forms:
\[
\langle \mathrm{D}\phi v_1,\mathrm{D}\phi v_2 \rangle_{g_2}=\langle  v_1, v_2 \rangle_{g_1}, \quad \text{ for all } v_1,v_2 \in T\Gamma_1,
\]
making $\phi$ a Riemannian isometry.
%
\end{proof}

In particular we get from this lemma that
$
\mathrm{D}\phi (P(\Gamma_1))=P(\Gamma_2).
$
Next we show that the mapping $\Psi$ takes the boundary of $M_1$ onto the boundary of $M_2$. In light of Proposition \ref{thm:bndry_defining} we need to understand how this map carries over the sets $\sigma(z,v)$, as in \eqref{eqn:sigma}. The following lemma gives an answer to this question.

\begin{lemma}
Let Riemannian manifolds $(M_1, g_1)$ and $(M_2, g_2)$ be as in Theorem \ref{thm:main}. If $(z_0,v) \in P(\Gamma_1)$ then 
$\Psi(\sigma(z_0,v)) = \sigma(\phi(z_0), \mathrm{D}\phi v)$. 
\label{lem:preserve_sigma}
\end{lemma}
\begin{proof}
Clearly we have that $\Psi(z_0)=\phi(z_0)\in \sigma(\phi(z_0), \mathrm{D}\phi v)$. So suppose that $p_0 \in \sigma(z_0,v)\setminus \{z_0\}$. 
Hence, by the same argument as in the proof of Lemma 
\ref{thm: Tzv=exit}, we get that $z_0$ is not in the cut-locus of $p_0$.
Thus by Proposition \ref{cor:smoothness_of_d} we can choose a neighborhood $U\times V\subset M_1 \times M_1$ of $(p_0,z_0)$ where the distance function $d_1(\cdot, \cdot)$ is smooth. 
Since the map $\Psi$ is a homeomorphism the set $\Psi(U)\subset M_2$ is open, and we have by \eqref{eq:distances_agree}
that for each $q \in \Psi(U)$, the function $r_{q}(\cdot) = d_1(\Psi^{-1}(q), \phi^{-1}(\cdot))$ is smooth on the open set $\phi(V\cap \Gamma_1)\subset \Gamma_2$.

Since $\phi\colon \Gamma_1 \to \Gamma_2$ is a Riemannian isometry we have that
\begin{equation}
\label{eq:pushforward_of_grad}
\grad_{\p M_2} r_{\Psi(p)}(\phi(z_0))
=\mathrm{D}\phi (z_0)\grad_{\p M_1} r_p(z_0), \quad \text{for } p \in U.
\end{equation}
Here $\mathrm{D}$ stands for the differential, $\grad_{\p M_1}$ for the boundary gradient of $\Gamma_1$ and $\grad_{\p M_2}$ for that of $\Gamma_2$.  Since the right hand side of equation \eqref{eq:pushforward_of_grad} is continuous in $p$, the function
\[
q=\Psi(p)\mapsto \grad_{\p M_2} r_{q}(\phi(z_0))
\]
is continuous in $\Psi(U)$. Finally
\[
\grad_{\p M_2} r_{\Psi(p_0)}(\phi(z_0))
=\mathrm{D}\phi (z_0)\grad_{\p M_1} r_{p_0}(z_0)=-\mathrm{D}\phi (z_0)v
\]
implies
$
\Psi(p_0) \in \sigma(\phi(z_0), \mathrm{D}\phi v).
$

On the other hand after reversing the roles of $M_1$ and $M_2$ we can use the same proof to show
$
\sigma(z_0,v) \supset \Psi^{-1}(\sigma(\phi(z_0), \mathrm{D}\phi v))$, 
implying
$\Psi(\sigma(z_0,v)) = \sigma(\phi(z_0), \mathrm{D}\phi v).
$
This ends the proof.
\end{proof}

\begin{lemma}
Let Riemannian manifolds $(M_1, g_1)$ and $(M_2, g_2)$ be as in Theorem \ref{thm:main}. Then $\Psi(\dM_1) = \dM_2$. Moreover, $\Psi(M_1^{int}) = M_2^{int}$.
\label{lem:boundary_to_boundary}
\end{lemma}
\begin{proof}
Let $p \in \dM_1$. Due to Proposition \ref{thm:bndry_defining} there is a $(z, v) \in 
P(\Gamma_1)$ such that $p$ is in the closed set $\sigma(z,v)$ and $r_{p_1}(z) = T_{z,v}$. Thus Lemma \ref{lem:preserve_sigma} gives 
$
\Psi(\sigma(z,v)) = \sigma(\phi(z),\mathrm{D}\phi v ),
$
and since $\Psi$ is a homeomorphism, also the set $\sigma(\phi(z),\mathrm{D}\phi v )$ is closed and contains $\Psi(p)$. Furthermore, by equation \eqref{eq:distances_agree} we have that
$
r_{\Psi(q)}(\phi(z))=r_q(z), \: \text{for all } \: q\in \sigma(z,v).
$
Therefore 
\[
T_{\phi(z),\mathrm{D}\phi v}=T_{z, v}=r_p(z)=r_{\Psi(p)}(\phi(z)).
\]
From here Proposition \ref{thm:bndry_defining} implies that $\Psi(p)$ is in $\p M_2$. Thus $\Psi(\p M_1)\subset \p M_2$ and by using the same argument for $\Psi^{-1}$ it follows that $\Psi(\dM_1) = \dM_2$. Since $M^{int}_1$ and $\p M_1$ are disjoint and $\Psi$ is a bijection we also have that $\Psi(M^{int}_1)=M^{int}_2$.
\end{proof}

\begin{lemma}
Let Riemannian manifolds $(M_1, g_1)$ and $(M_2, g_2)$ be as in Theorem \ref{thm:main}.  The mapping $\Psi:M_1 \to M_2$, given in formula \eqref{eqn:diffeo_between_manifolds}, is a diffeomorphism. 
\label{lemma:diffeo_between_manifolds}
\end{lemma}

\begin{proof} 
Let $p_0 \in M_1$, and choose $W_{p_0}\subset \p M_1$ as in Theorem \ref{prop:dist_smoothness}. Since $\phi \colon \Gamma_1 \to \Gamma_2$ is a diffeomorphism, the set $\phi(W_{p_0}\cap \Gamma_1)$ is open and dense in $\Gamma_2$. Then for $\Psi(p_0)\in M_2$ we choose $W_{\Psi(p_0)}\subset \p M_2$ as in Theorem \ref{prop:dist_smoothness} and consider the non-empty open set $W_{\Psi(p_0)}\cap \phi(W_{p_0}\cap \Gamma_1)\subset \Gamma_2$. We pick $z_0 \in W_{p_0}\cap \Gamma_1$ such that $\phi(z_0)\in W_{\Psi(p_0)}\cap \phi(W_{p_0}\cap \Gamma_1)$. 

Let neighborhoods $\upo\subset M_1$ of $p_0$ and $\vzo\subset M_1$ of $z_0$ be such that the distance function $d_1(\cdot,\cdot)$ is smooth in the product set $\upo\times \vzo$. We also choose neighborhoods $U_{\Psi(p_0)}\subset M_2$ of $\Psi(p_0)$ and $V_{\Psi(p_0)}\subset M_2$ of $\phi(z_0)=\Psi(z_0)$ to be such that the distance function $d_2(\cdot,\cdot)$ is smooth in the product set $U_{\Psi(p_0)}\times V_{\Psi(p_0)}$. Since $\Psi\colon M_1 \to M_2$ is a homeomorphism we may choose these four sets in such a way that they satisfy
\[
\Psi(U_{p_0})=U_{\Psi(p_0)},\quad \text{ and } \quad \Psi(V_{p_0})=V_{\Psi(p_0)}.
\]
By Lemma \ref{lem:boundary_to_boundary} we know that $\Psi(p_0)\in  M_2^{int}$ if and only if $p_0 \in M_1^{int}$, and $\Psi(p_0)\in \p M_2$ if and only if $p_0 \in \p M_1$. Next we consider the interior and boundary cases separately. 

\medskip
Suppose first that $p_0$ is an interior point of $M_1$. The functions 
\[
\upo\ni p \mapsto \alpha_1(p)=(-\grad_{\p M_1}r_p(z_0) ,  r_p(z_0)) \in P_{z_0}(\Gamma_1) \times \R, 
\]
and
\[
U_{\Psi(p_0)}\ni q \mapsto \alpha_2(q)=(-\grad_{\p M_2}r_q(\phi(z_0)) ,  r_q(\phi(z_0)))\in P_{\phi(z_0)}(\Gamma_2)\times \R, 
\]
as in Proposition \ref{thm:int_coords}, are smooth local coordinate maps of $M_1$ and $M_2$ respectively. Moreover, by the computations done in the proof of Lemma \ref{lem:preserve_sigma} we get for every $p \in \upo$ that
\[
r_p(z_0)=r_{\Psi(p)}(\phi(z_0)), \quad \text{and} \quad \mathrm{D} \phi(z_0) \grad_{\p M_1}r_p(z_0)=\grad_{\p M_2} r_{\Psi(p)}(\phi(z_0)).
\]
Therefore for any $(v,t)\in \alpha_1(\upo)$ we have that
\[
(\alpha_2\circ \Psi\circ \alpha^{-1}_1)(v,t)=(\mathrm{D}\phi(z_0)v,t).
\]
Thus we have proven that the map $\alpha_2\circ \Psi\circ \alpha^{-1}_1\colon \alpha_1(\upo) \to \alpha_2(U_{\Psi(p_0)})$ is smooth.

\medskip
Then we let $p_0$ be a boundary point of $M_1$. Let $\eta_1(p):=-\grad_{\p M_1}r_p(z_0)$ for $p \in \upo$ and choose $\upo' \subset \upo$ as in Lemma \ref{thm:bndry_tzv=texit} to be such that the set $\sigma(z_0,\eta_1(p))$ is closed and the function $p\mapsto T_{z_0,\eta_1(p)}$ is smooth for every $p \in \upo'$. Let $U'_{\Psi(p_0)}:=\Psi(\upo')\subset U_{\Psi(p_0)}$ and denote $\eta_2(q):=-\grad_{\p M_2}r_{q}(\phi(z_0))$ for $q \in U'_{\Psi(p_0)}$. Since we have that $\mathrm{D}\phi(z_0) \eta_1(p) = \eta_2(\Psi(p))$ it holds by Lemma \ref{lem:preserve_sigma} that the set
$
\sigma(\phi(z_0),\eta_2(\Psi(p)))=\Psi(\sigma(z_0,\eta_1(p))),
$
is closed for every $p \in \upo'$, and thus the function $U'_{\Psi(p_0)}\ni q\to T_{\phi(z_0),\eta_2(q)}$ is smooth by Corollary \ref{cor:cut=Tzv}. Moreover, we have $T_{z_0,\eta_1(p)}=T_{\phi(z_0),\eta_2(\Psi(p))}$ for every $p \in \upo'$.

Then we consider local coordinate maps
\[
\upo' \ni p \mapsto \beta_1(p)=(\eta_1(p), T_{z_0,\eta_1(p)}- r_p(z_0)) \in P_{z_0}(\Gamma_1) \times [0,\infty), 
\]
of $M_1$ and
\[
U'_{\Psi(p_0)}\ni q \mapsto \beta_2(q)=(\eta_2(q) ,  T_{\phi(z_0),\eta_2(q)} -r_q(\phi(z_0)))\in  P_{\phi(z_0)}(\Gamma_2) \times [0,\infty), 
\]
of $M_2$, as in Proposition \ref{thm:bndry_coords}. By the discussion above we have for any $(v,t)\in \beta_1(\upo')$ that
\[
(\beta_2\circ \Psi\circ \beta_1^{-1})(v,t)=(\mathrm{D}\phi(z_0) v, t),
\]
which implies that the map $(\beta_2\circ \Psi\circ \beta_1^{-1})\colon \beta_1(\upo')\to \beta_{2}(U'_{\Psi(p_0)})$ is smooth. 

\medskip
By combining these two cases we have proved that for every $p_0 \in M$ a local representation of the map $\Psi$ is smooth, making $\Psi\colon M_1 \to M_2$ smooth. Finally by an analogous argument for $\Psi^{-1}$ we can show that this map is also smooth. Thus $\Psi\colon M_1 \to M_2$ is a diffeomorphism as claimed. 
\end{proof}
We are ready to present the proof of our main inverse problem:

\begin{proof}[Proof of Theorem \ref{thm:main}]
By Lemma \ref{lemma:diffeo_between_manifolds} we know that the map $\Psi\colon M_1 \to M_2$ is a diffeomorphism. We define a metric tensor $\tilde{g}_2$ on $M_1$ as the pull back of the metric $g_2$ with respect to map $\Psi$. Thus it suffices to consider a smooth manifold $M=M_1$ with an open measurement region $\Gamma=\Gamma_1\subset \p M$ and two Riemannian metrics $g_1$ and $\tilde{g_2}$. Moreover $\p M$ is strictly convex with respect to both of these metrics.

Let $\tilde d_2(\cdot,\cdot)$ be the distance function of $\tilde{g_2}$. We note that due to equation \eqref{eq:distances_agree} we have
$
d_1(p,z)
=\tilde{d}_2(p,z), \: \text{for all } (p,z)\in M\times \Gamma.
$
By Lemma \ref{lemma:iso_on_gammas} we get that $g_1(p)=\tilde g_2(p)$ for all $p \in \Gamma$. Let $p_0 \in M$. Thus the map $H_{p_0}$ given by \eqref{eq:grad_map} is the same for both metrics. From here Lemma \ref{lem:open_subset_unit_sphere} and Proposition \ref{prop:metric} imply that $g_1(p_0)=\tilde g_2(p_0)$. Therefore map $\Psi$ is a Riemannian isometry as claimed.
\end{proof}

\bibliographystyle{abbrv}
\bibliography{biblio,introduction}
\end{document}